\documentclass[11pt,a4paper]{article}
\usepackage{amssymb}
\usepackage{amsmath}
\usepackage{amsfonts}
\usepackage{bbm}
\usepackage{amsthm}
\usepackage{mathrsfs}
\usepackage{hyperref}
\usepackage{color}
\usepackage[margin=2.41cm]{geometry}
\usepackage[all,cmtip,2cell]{xy}
\UseAllTwocells
\usepackage[utf8]{inputenc}
\usepackage{graphicx}
\usepackage{varwidth}
\usepackage{comment}

\usepackage{upgreek}
\usepackage{rotating}

\usepackage{tikz}
\usetikzlibrary{shapes.geometric}

\definecolor{darkred}{rgb}{0.8,0.1,0.1}
\hypersetup{
     colorlinks=false,         
     linkcolor=darkred,
     citecolor=blue,
}

\theoremstyle{plain}
\newtheorem{theo}{Theorem}[section]
\newtheorem{lem}[theo]{Lemma}
\newtheorem{propo}[theo]{Proposition}
\newtheorem{cor}[theo]{Corollary}

\theoremstyle{definition}
\newtheorem{defi}[theo]{Definition}

\newtheorem{exx}[theo]{Example}
\newenvironment{ex}
  {\pushQED{\qed}\exx}
  {\popQED\endexx}

\newtheorem{remm}[theo]{Remark}
\newenvironment{rem}
  {\pushQED{\qed}\remm}
  {\popQED\endremm}

\numberwithin{equation}{section}

\def\nn{\nonumber}

\def\bbC{\mathbb{C}}

\def\Hom{\mathrm{Hom}}

\def\id{\mathrm{id}}
\def\ID{\mathrm{Id}}

\def\1{I}
\def\oone{\mathbbm{1}}
\def\op{\mathrm{op}}

\def\triv{\operatorname{triv}}

\def\Loc{\mathbf{Loc}}

\def\astObj{{\ast{\text{-}}\mathbf{Obj}}}

\def\Set{\mathbf{Set}}
\def\Alg{\mathbf{Alg}}
\def\Vec{\mathbf{Vec}}
\def\Ch{\mathbf{Ch}}
\def\Mon{\mathbf{Mon}}
\def\CMon{\mathbf{CMon}}
\def\astMon{{\ast{\text{-}}\mathbf{Mon}}}
\def\astAlg{{\ast{\text{-}}\mathbf{Alg}}}

\def\CC{\mathbf{C}}
\def\DD{\mathbf{D}}

\def\EE{\mathbf{E}}

\def\MM{\mathbf{M}}

\def\Cat{\mathbf{Cat}}
\def\OCat{\mathbf{OrthCat}}
\def\ICat{\mathbf{ICat}}
\def\MonCat{\mathbf{MCat}}
\def\SMonCat{\mathbf{SMCat}}
\def\IMonCat{\mathbf{IMCat}}
\def\ISMonCat{\mathbf{ISMCat}}

\def\sSet{\mathbf{sSet}}
\def\SymSeq{\mathbf{SymSeq}}
\def\Op{\mathbf{Op}}
\def\astOp{{\ast{\text{-}}}\mathbf{Op}}

\def\AAA{\mathfrak{A}}

\def\CCC{\mathfrak{C}}
\def\DDD{\mathfrak{D}}

\def\O{\mathcal{O}}
\def\P{\mathcal{P}}
\def\E{\mathcal{E}}

\newcommand\und[1]{\underline{#1}}
\newcommand\ovr[1]{\overline{#1}}

\DeclareMathOperator*{\Motimes}{\text{\raisebox{0.25ex}{\scalebox{0.8}{$\bigotimes$}}}}

\def\sk{\vspace{2mm}}

\makeatletter
\let\@fnsymbol\@alph
\makeatother

\title{%
Involutive categories, colored $\ast$-operads and quantum field theory
}

\author{%
Marco Benini$^{1,a}$, 
Alexander Schenkel$^{2,b}$\ and\
Lukas Woike$^{1,c}$\vspace{4mm}\\
{\small ${}^1$ Fachbereich Mathematik, Universit\"at Hamburg,}\\
{\small Bundesstr.~55, 20146 Hamburg, Germany.}\vspace{3mm}\\
{\small ${}^2$ School of Mathematical Sciences, University of Nottingham,}\\
{\small University Park, Nottingham NG7 2RD, United Kingdom.}\vspace{5mm}\\
{\small \begin{tabular}{ll}
Email: & ${}^a$~\texttt{marco.benini@uni-hamburg.de}\\
& ${}^b$~\texttt{alexander.schenkel@nottingham.ac.uk}\\
& ${}^c$~\texttt{lukas.jannik.woike@uni-hamburg.de}\vspace{3mm}
\end{tabular}
}
}

\date{January 2019}

\begin{document}

\maketitle

\begin{abstract}
\noindent 
Involutive category theory provides a flexible framework to describe involutive structures on algebraic objects, such as anti-linear involutions on complex vector spaces. Motivated by the prominent role of involutions in quantum (field) theory, we develop the involutive analogs of colored operads and their algebras, named colored $\ast$-operads and $\ast$-algebras. Central to the definition of colored $\ast$-operads is the involutive monoidal category of symmetric sequences, which we obtain from a general product-exponential $2$-adjunction whose right adjoint forms involutive functor categories. For $\ast$-algebras over $\ast$-operads we obtain involutive analogs of the usual change of color and operad adjunctions. As an application, we turn the colored operads for algebraic quantum field theory into colored $\ast$-operads. The simplest instance is the associative $\ast$-operad, whose $\ast$-algebras are unital and associative $\ast$-algebras.
\end{abstract}

\vspace{2mm}

\paragraph*{Report no.:} ZMP-HH/18-6, Hamburger Beitr\"age zur Mathematik Nr.\ 725

\paragraph*{Keywords:} involutive categories, involutive monoidal categories, $\ast$-monoids, colored operads, $\ast$-algebras, algebraic quantum field theory

\paragraph*{MSC 2010:} 18Dxx, 81Txx

\tableofcontents




\section{\label{sec:intro}Introduction and summary} 
In ordinary category theory,  an involution on an object $c\in\CC$
of a category $\CC$ is an endomorphism $i : c\to c$ 
that squares to the identity, i.e.\ $i^2 =\id_c$. 
Unfortunately, this concept is too rigid to describe many examples
of interest. For instance, given an associative and unital $\ast$-algebra $A$
over $\bbC$, e.g.\ the algebra of observables of a quantum system, 
the involution $\ast : A \to A$ on its underlying vector space is {\em not} 
an endomorphism in the category of complex vector spaces, but rather 
a complex {\em anti-linear} map. 
\sk

Involutive categories \cite{BeggsMajid,Egger,Jacobs} were developed 
in order to introduce the flexibility required to resolve this insufficiency. 
Their definition is a particular instance of 
the ``microcosm principle'' of Baez and Dolan
\cite{BaezDolan}, which states that {\em certain algebraic structures can be 
defined in any category equipped with a categorified version of the same structure}.
Hence, an involutive category is a category $\CC$ equipped with
an endofunctor $J :\CC\to\CC$ that squares to the identity endofunctor
$\ID_\CC$, up to a given natural isomorphism $j : \ID_\CC \to J^2$ which has to
satisfy certain coherence conditions (cf.\ Definition \ref{def:invcat}).
In an involutive category $(\CC,J,j)$, one can introduce
a more flexible concept of involution on an object $c\in\CC$, which is
given by a $\CC$-morphism $\ast : c\to Jc$ satisfying
$(J\ast)~\ast = j_c$ as morphisms from $c$ to $J^2c$ (cf.\ Definition \ref{def:astobject}).
Such objects (homotopy fixed points, as a matter of fact) are called 
self-conjugates in \cite{Jacobs}, involutive objects in \cite{Egger} 
and $\ast$-objects in \cite{BeggsMajid}. We shall follow the latter terminology
because it seems the most natural one to us.
If a category is equipped with its trivial involutive structure $J = \ID_\CC$ and $j=\id_{\ID_\CC}$
(cf.\ Example \ref{ex:trivial}), then $\ast$-objects are just endomorphisms squaring to the identity, 
i.e.\ the ordinary involutions mentioned above. This framework, however,
becomes much richer and flexible by allowing for non-trivial involutive structures:
For example, endowing the category of complex vector spaces $\Vec_{\bbC}$
with the involutive structure given by the endofunctor that assigns to
a complex vector space $V$ its complex conjugate vector space $\ovr{V}$,
the complex anti-linear map underlying a $\ast$-algebra
may be regarded as a $\ast$-object $\ast : A\to\ovr{A}$ in this involutive category
(cf.\ Examples \ref{ex:vec} and \ref{ex:vec2}).
\sk

The observables of a quantum system form a unital and associative $\ast$-algebra over $\bbC$. 
This shows the relevance of involutive categories for general quantum theory, 
quantum field theory and also noncommutative geometry. 
Our main motivation for this paper stems precisely from these areas and more specifically from 
our recent operadic approach to algebraic quantum field theory
\cite{BeniniSchenkelWoike}. There the axioms of algebraic quantum field theory \cite{HaagKastler,Brunetti}
are encoded in a colored operad and generalized to richer target categories, 
such as chain complexes and other symmetric monoidal categories, 
which are central in modern approaches to quantum gauge theories 
\cite{CostelloGwilliam,BeniniSchenkelSzabo,BeniniSchenkel,BeniniSchenkelWoike,BeniniSchenkelWoikehomotopy,YauQFT}. 
For their physical interpretation, however, it is essential that 
quantum systems such as quantum field theories come equipped with involutions. 
These enable us to perform the GNS construction and recover the usual
probabilistic interpretation of quantum theory.
We refer to \cite{Jacobs} for a generalization of the GNS construction 
to involutive symmetric monoidal categories.
\sk
 
The purpose of this paper is to combine the theory of colored operads 
and that of involutive categories, resulting in what we shall call colored $\ast$-operads. 
Despite of our quite concrete motivation, we believe that
working out the theory of colored $\ast$-operads in full generality 
provides an interesting and valuable addition
to the largely unexplored field of involutive category theory. On the one hand, 
our constructions naturally lead to interesting new structures such as involutive 
functor categories, which have not been discussed in the literature. 
On the other hand, our study of involutive structures on the category of 
symmetric sequences, which is a monoidal category that does not admit a braiding,
provides an interesting example of an involutive monoidal
category in the sense of \cite{Jacobs}, but not in the sense of \cite{BeggsMajid,Egger},
see Remark \ref{rem:whynonrev} for details. 
This shows that Jacobs' definition of involutive monoidal categories 
is the one suitable to develop the theory of colored $\ast$-operads, 
consequently we shall use this one in our paper.
\sk

The outline of the paper is as follows:
Sections \ref{sec:invcat} and \ref{sec:invmoncat}
contain a brief review of involutive categories and 
involutive (symmetric) monoidal categories following
mostly \cite{Jacobs}. We shall in particular
emphasize and further develop the $2$-categorical
aspects of this theory, including the $2$-functorial
behavior of the assignments of the categories of
$\ast$-objects and $\ast$-monoids. For the sake of concreteness, 
we also describe the most relevant constructions 
and definitions arising this way in fully explicit terms. 
Theorems \ref{theo:iso2trivial} and \ref{theo:iso2trivialmonoidal} establish 
simple criteria that are useful
to detect whether an involutive ((symmetric) monoidal) category
is isomorphic to one with a trivial involutive structure.
In Section \ref{sec:symseq} we show that the category
of colored symmetric sequences, which underlies colored operad theory,
carries a canonical involutive monoidal structure in the sense of
\cite{Jacobs}, but not in the sense of \cite{BeggsMajid,Egger}.
The relevant involutive structure is obtained by employing 
a general construction, namely exponentiation of involutive
categories, which results in involutive structures on functor categories.
Colored $\ast$-operads with values in any cocomplete involutive closed
symmetric monoidal category $(\MM,J,j)$ are defined in Section \ref{sec:astOp}
as $\ast$-monoids in our involutive monoidal category of colored symmetric
sequences. In Proposition \ref{propo:OpastObj} we shall prove that the resulting category
is isomorphic to the category of ordinary colored operads with values in the category
of $\ast$-objects in $(\MM,J,j)$, which provides an alternative point of view on 
colored $\ast$-operads. The possibility to switch between these equivalent 
perspectives is useful for concrete applications and also to import techniques from ordinary 
operad theory to the involutive setting. In Section \ref{sec:astAlg} we introduce and 
study the category of $\ast$-algebras over colored $\ast$-operads.
In particular, we prove that a change of colored $\ast$-operad induces
an adjunction between the associated categories
of $\ast$-algebras, which generalizes the corresponding
crucial and widely used result from ordinary to involutive category theory.
Finally, in Section \ref{sec:QFTs} we endow the algebraic quantum field theory
operads constructed in \cite{BeniniSchenkelWoike} with a canonical order-reversing structure 
of colored $\ast$-operads and provide a characterization of
the corresponding categories of $\ast$-algebras. As a simple example,
we obtain a $\ast$-operad structure on the associative operad
and show that its $\ast$-algebras behave like $\ast$-algebras over $\bbC$ 
in the sense that the involution reverses the order of multiplication 
$(a\,b)^\ast = b^\ast\,a^\ast$. It is essential to emphasize
that this order-reversal is encoded in our $\ast$-operad structure. 
This is radically different from the approach of \cite{BeggsMajid,Egger}, 
whose definition of an involutive monoidal category prescribes that 
the endofunctor $J$ reverses the monoidal structure up to natural isomorphism, 
thus recovering unital and associative $\ast$-algebras over $\bbC$ 
directly as $\ast$-monoids in $\Vec_\bbC$.

\paragraph*{Notations:} We denote categories by boldface letters like $\CC$, $\DD$ and $\EE$.
Objects in categories are indicated by $c\in \CC$ and we write
$\CC(c,c^\prime)$ for the set of morphisms from $c$ to $c^\prime$ in $\CC$. Functors are denoted 
by capital letters like $F : \CC\to \CC^\prime$ or  $X : \DD\to\CC$, and so are the identity functors
$\ID_\CC^{} : \CC\to\CC$. Natural transformations are denoted by Greek letters like 
$\zeta : F\to G$ or $\alpha : X\to Y$. Given functors $K : \DD^\prime \to \DD$, $X : \DD\to \CC$ 
and $J:\CC\to \CC^\prime$, we denote their composition simply by juxtaposition $JXK : \DD^\prime \to \CC^\prime$.
Given also a natural transformation $\alpha : X\to Y$ of functors $X,Y: \DD\to \CC$,
we denote by
\begin{subequations}
\begin{flalign}
J\alpha K \,:\, JXK ~\longrightarrow ~ JYK\quad  
\end{flalign}
the {\em whiskering} of $J$, $\alpha$ and $K$. Explicitly,
$J\alpha K$ is the natural transformation with components
\begin{flalign}
(J\alpha K)_{d^\prime}^{} = J\alpha_{K d^\prime}^{}\,:\, JXKd^\prime ~\longrightarrow ~ JYKd^\prime\quad,
\end{flalign}
\end{subequations}
for all $d^\prime\in\DD^\prime$. For $\beta : Y\to Z$ another natural transformation, one easily confirms that
\begin{flalign}
(J\beta K)~(J\alpha K) \,=\, J\big(\beta \alpha \big)K\,:\, JXK~\longrightarrow ~ JZK \quad,
\end{flalign}
where (vertical) composition of natural transformations is also denoted by juxtaposition.
We shall need some basic elements of (strict) $2$-category theory, for which we refer to \cite{KellyStreet}.


\section{\label{sec:invcat}Involutive categories}
This section contains a brief review of involutive categories. 
We shall mostly follow the definitions and conventions of Jacobs \cite{Jacobs} 
and refer to this paper for more details and some of the proofs.
We strongly emphasize and also develop further the $2$-categorical aspects 
of involutive category theory established in \cite{Jacobs}, which 
will be relevant for the development of our present paper. 
When it comes to notations and terminology, we sometimes
prefer the work of Beggs and Majid \cite{BeggsMajid} 
and the one of Egger \cite{Egger}.

\subsection{Basic definitions and properties}
\begin{defi}\label{def:invcat}
An {\em involutive category} is a triple $(\CC,J,j)$ consisting of
a category $\CC$, an endofunctor $J : \CC\to\CC$ and a natural isomorphism
$j : \ID_\CC^{} \to J^2$ satisfying 
\begin{flalign}
 j J \,=\, J j \,:\, J ~\longrightarrow ~ J^3\quad.
\end{flalign}
\end{defi}

\begin{ex}\label{ex:trivial}
For any category $\CC$, the triple $(\CC,\ID_{\CC}^{},\id_{\ID_{\CC}})$
defines an involutive category. We call this the {\em trivial involutive category} over $\CC$.
\end{ex}

\begin{ex}\label{ex:vec}
Let $\Vec_\bbC$ be the category of complex vector spaces.
Consider the endofunctor $\overline{(-)} :\Vec_{\bbC}\to \Vec_{\bbC}$ 
that assigns to any $V\in\Vec_{\bbC}$ its complex conjugate vector
space $\overline{V}\in\Vec_{\bbC}$ and to any $\bbC$-linear map $f : V\to W$
the canonically induced $\bbC$-linear map $\overline{f} : \overline{V}\to\overline{W}$.
Notice that $\overline{\overline{(-)}}= \ID_{\Vec_{\bbC}}$, hence
the triple $(\Vec_\bbC,\overline{(-)},\id_{\ID_{\Vec_{\bbC}}})$ is an involutive category. 
\end{ex}

\begin{ex}\label{ex:Cprofiles}
Let $\CCC$ be any non-empty set and $\Sigma_{\CCC}$ the associated {\em groupoid of $\CCC$-profiles}.
The objects of $\Sigma_{\CCC}$  are finite sequences $\und{c} = (c_1,\dots,c_n)$
of elements in $\CCC$, including also the empty sequence $\emptyset \in \Sigma_{\CCC}$.
We denote by $\vert \und{c}\vert =n$ the length of the sequence.
The morphisms of $\Sigma_{\CCC}$ are right permutations
$\sigma : \und{c} \to \und{c}\sigma := (c_{\sigma(1)},\dots,c_{\sigma(n)})$,
with $\sigma\in\Sigma_{\vert \und{c}\vert}$ in the symmetric group on $\vert \und{c}\vert$ letters.
We define an endofunctor $\mathrm{Rev} : \Sigma_\CCC\to \Sigma_\CCC$
as follows: To an object $\und{c}= (c_1,\dots,c_n)\in\Sigma_\CCC$ 
it assigns the reversed sequence
\begin{subequations}
\begin{flalign}
\mathrm{Rev}(\und{c}) \,:= \, \und{c}\, \rho^{}_{\vert\und{c}\vert} \,:= \, (c_n,\dots, c_1)\quad,
\end{flalign}
where $\rho^{}_{\vert\und{c}\vert}\in\Sigma_{\vert\und{c}\vert}$ 
denotes the order-reversal permutation. To a 
$\Sigma_\CCC$-morphism $\sigma : \und{c} \to \und{c}\sigma$  it assigns the
right permutation
\begin{flalign}
\mathrm{Rev}(\sigma) :=  \rho^{}_{\vert\und{c}\vert}\, \sigma\,\rho^{}_{\vert\und{c}\vert} \,:\,  \mathrm{Rev}(\und{c}) ~\longrightarrow ~ \mathrm{Rev}(\und{c}\sigma) \quad,
\end{flalign}
\end{subequations}
where we also used that $\vert \und{c}\sigma\vert = \vert \und{c}\vert$.
Notice that $\mathrm{Rev}^2 = \ID_{\Sigma_\CCC}$, hence
the triple $(\Sigma_\CCC,\mathrm{Rev},\id_{\ID_{\Sigma_\CCC}})$ is an involutive category.
\end{ex}

The following very useful result appears in \cite[Lemma 1]{Jacobs}.
\begin{lem}\label{lem:Jselfadjoint}
For every involutive category $(\CC,J,j)$, the endofunctor $J: \CC\to \CC$ is self-adjoint, i.e.\  $J \dashv J$.
As a consequence, $J$ preserves all limits and colimits that exist in $\CC$.
\end{lem}

\begin{defi}\label{def:invfunandnat}
An {\em involutive functor}  $(F,\nu) : (\CC,J,j)\to (\CC^\prime ,J^\prime ,j^\prime) $ consists of
a functor $F : \CC\to\CC^\prime$ and a natural transformation $\nu : F J \to J^\prime F$
satisfying
\begin{flalign}\label{eqn:invfundiagram}
\xymatrix{
\ar[d]_-{Fj} F \ar@{=}[rr]& & F\ar[d]^-{j^\prime F} \\
F J^2 \ar[r]_-{\nu J }& J^\prime F J  \ar[r]_-{J^\prime \nu} &J^{\prime 2}F
}
\end{flalign}
An {\em involutive natural transformation} 
$\zeta : (F,\nu) \to (G,\chi)$ between involutive functors $(F,\nu), (G,\chi) :
(\CC,J,j)\to (\CC^\prime ,J^\prime ,j^\prime)$ is a natural transformation $\zeta : F\to G$ satisfying
\begin{flalign}
\xymatrix{
\ar[d]_-{\nu} FJ  \ar[rr]^-{\zeta J }&& GJ \ar[d]^-{\chi}\\
J^\prime F \ar[rr]_-{J^\prime \zeta}&& J^\prime G
}
\end{flalign}
\end{defi}

\begin{propo}
Involutive categories, involutive functors and involutive natural transformations
form a $2$-category $\ICat$.
\end{propo}

\begin{rem}\label{rem:2cat}
Let us describe the $2$-category structure on $\ICat$ explicitly.
\begin{itemize}
\item[(i)] For any involutive category $(\CC,J,j)$, the identity involutive functor
is given by  $\ID_{(\CC,J,j)}^{} := (\ID_\CC, \id_{J}) :  (\CC,J,j)\to  (\CC,J,j)$.

\item[(ii)] Given two involutive functors $(F,\nu) : (\CC,J,j)\to (\CC^\prime , J^\prime , j^\prime)$ 
and $(F^\prime ,\nu^\prime): (\CC^\prime , J^\prime , j^\prime)\to (\CC^{\prime\prime} , J^{\prime\prime} , 
j^{\prime\prime})$, their composition is given by
\begin{flalign}
(F^\prime ,\nu^\prime )\, (F,\nu)  := \big(F^\prime F , (\nu^\prime F)\, (F^\prime \nu) \big) 
\, :\, (\CC,J,j)~\longrightarrow ~ (\CC^{\prime\prime} , J^{\prime\prime} , j^{\prime\prime})\quad.
\end{flalign}

\item[(iii)] Vertical/horizontal composition of involutive natural transformations
is given by vertical/horizontal composition of their underlying natural transformations.
(It is easy to verify that the latter compositions define involutive natural transformations.)\qedhere
\end{itemize}
\end{rem}

The following technical lemma is proven in \cite[Lemma 2]{Jacobs}.
\begin{lem}\label{lem:nuiso}
For every involutive functor $(F,\nu) : (\CC,J,j)\to (\CC^\prime ,J^\prime ,j^\prime)$, 
the natural transformation $\nu: F J \to J^\prime F$ is a natural isomorphism.
\end{lem}

As in any $2$-category, there exists 
the concept of adjunctions in the $2$-category $\ICat$.
\begin{defi}\label{def:invadj}
An {\em involutive adjunction} 
\begin{flalign}
\xymatrix{
(L,\lambda) \,:\, (\CC,J,j)~\ar@<0.5ex>[r]&\ar@<0.5ex>[l]  ~(\DD,K,k) \,:\, (R,\rho)
}
\end{flalign}
consists of two involutive functors $(L,\lambda) : (\CC,J,j)\to (\DD,K,k)$
and $(R,\rho) : (\DD,K,k)\to (\CC,J,j)$ together with two involutive
natural transformations $\eta : \ID_{(\CC,J,j)} \to (R,\rho)\,(L,\lambda)$
(called {\em unit}) and $\epsilon : (L,\lambda)\,(R,\rho) \to \ID_{(\DD,K,k)}$
(called {\em counit}) that satisfy the triangle identities
\begin{flalign}
\xymatrix{
\ar[drr]_-{\id_{(R,\rho)}} (R,\rho) \ar[rr]^-{\eta\,(R,\rho)}&& (R,\rho)\,(L,\lambda)\,(R,\rho)\ar[d]^-{(R,\rho)\,\epsilon}\\
&& (R,\rho)
}\qquad\qquad
\xymatrix{
\ar[drr]_-{\id_{(L,\lambda)}}(L,\lambda) \ar[rr]^-{(L,\lambda)\,\eta} && (L,\lambda)\,(R,\rho)\,(L,\lambda)
\ar[d]^-{\epsilon\,(L,\lambda)}\\
&& (L,\lambda)
}
\end{flalign}
We also denote involutive adjunctions simply by $(L,\lambda) \dashv (R,\rho)$.
\end{defi}

\begin{rem}
Applying the forgetful $2$-functor $\ICat\to \Cat$, every involutive adjunction
$(L,\lambda)\dashv (R,\rho)$ defines an ordinary adjunction $L\dashv R$ in the 
$2$-category of categories $\Cat$. Notice that an involutive adjunction is the same thing as
an ordinary adjunction $L\dashv R$ (between categories equipped with an involutive structure) 
whose functors $L$ and $R$ are equipped
with involutive structures that are compatible with the unit and counit 
in the sense that the latter become of involutive natural transformations.
This alternative point of view will be useful 
in Corollary \ref{cor:coloradjunction} and Theorem \ref{theo:adjunctionastAlg} below, 
where we make use of the construction in the following proposition.
\end{rem}

\begin{propo}\label{propo:invadj}
Let $(R,\rho): (\DD,K,k) \to (\CC,J,j)$ be an involutive functor 
and suppose that $L: \CC \to \DD$ is a left adjoint to the functor 
$R: \DD \to \CC$. Define a natural transformation 
$\lambda$ by
\begin{flalign}
\xymatrix{
\ar[d]_-{LJ \eta} LJ \ar[rr]^-{\lambda} && KL\\
LJRL \ar[rr]_-{L \rho^{-1} L} && LRKL\ar[u]_-{\epsilon KL}
}
\end{flalign}
where $\eta: \ID_{\CC} \to RL$ and $\epsilon: LR \to \ID_{\DD}$ 
are the unit and counit of the adjunction $L \dashv R$.
Then $(L,\lambda) \dashv (R,\rho)$ is an involutive adjunction.
\end{propo}
\begin{proof}
The above diagram defines a natural transformation $\lambda$ because
$\rho$ is a natural isomorphism, cf.\ Lemma \ref{lem:nuiso}.
A slightly lengthy diagram chase shows that $(L,\lambda) : (\CC,J,j)\to (\DD,K,k)$
is an involutive functor. Furthermore, by the definition of 
$\lambda$, the natural transformations $\eta$ and $\epsilon$ 
are involutive natural transformations. 
\end{proof}

\begin{rem}
Even though we will not need it in the following, let us briefly mention that
the dual of Proposition \ref{propo:invadj} also holds true: Let $(L,\lambda) : (\CC,J,j)\to (\DD,K, k)$
be an involutive functor and suppose that $R :\DD\to \CC$ is a right adjoint to the functor
$L : \CC\to\DD$. Then $(L,\lambda)\dashv (R,\rho)$ is an involutive adjunction for
$\rho$ defined by 
\begin{flalign}
\xymatrix{
\ar[d]_-{\eta JR} JR \ar[rr]^-{\rho^{-1}} && RK\\
RLJR \ar[rr]_-{R \lambda R} && RKLR\ar[u]_-{RK\epsilon}
}
\end{flalign}
where $\eta: \ID_{\CC} \to RL$ and $\epsilon: LR \to \ID_{\DD}$ 
are the unit and counit of the adjunction $L \dashv R$.
\end{rem}

\subsection{$\ast$-objects}
\begin{defi}\label{def:astobject}
A {\em $\ast$-object} in an involutive category 
$(\CC,J,j)$ is a $\CC$-morphism $\ast : c \to Jc$ satisfying
\begin{flalign}
\xymatrix{
\ar[drr]_-{j_c^{}} c \ar[rr]^-{\ast} && Jc\ar[d]^-{J\ast}\\
&& J^2c
}
\end{flalign}
A {\em $\ast$-morphism} $f: (\ast : c \to Jc) \to (\ast^\prime : c^\prime \to Jc^\prime)$
is a $\CC$-morphism $f :c\to c^\prime$ satisfying
\begin{flalign}
\xymatrix{
\ar[d]_-{\ast}c\ar[rr]^-{f} && c^\prime\ar[d]^-{\ast^\prime}\\
Jc\ar[rr]_-{Jf} && Jc^\prime
}
\end{flalign} 
We denote the category of $\ast$-objects in $(\CC,J,j)$
by $\astObj(\CC,J, j)$.
\end{defi}

\begin{rem}\label{rem:astisiso}
For any $\ast$-object $(\ast : c\to Jc) \in \astObj(\CC,J, j)$, 
the $\CC$-morphism $\ast : c\to Jc$ is an isomorphism with inverse given by
$ j_c^{-1}\, J\ast : Jc\to c$.
\end{rem}

\begin{ex}\label{ex:trivial2}
Consider the trivial involutive category $(\CC,\ID_\CC,\id_{\ID_\CC})$ 
from Example \ref{ex:trivial}. A $\ast$-object consists of an object $c\in\CC$ equipped
with a $\CC$-endomorphism $\ast: c \to c$ satisfying $\ast^2 =\id_c$, i.e.\ an object equipped with an involution.
\end{ex}

\begin{ex}\label{ex:vec2}
Consider the involutive category $(\Vec_\bbC,\overline{(-)},\id_{\ID_{\Vec_\bbC}})$
from Example \ref{ex:vec}. A $\ast$-object consists of a complex vector space $V$
equipped with a complex {\em anti-linear} map $\ast: V \to V$ 
satisfying $\ast^2=\id_V$. 
\end{ex}

\begin{ex}\label{ex:Cprofiles2}
Consider the involutive category $(\Sigma_\CCC,\mathrm{Rev},\id_{\ID_{\Sigma_\CCC}})$
from Example \ref{ex:Cprofiles}. A $\ast$-object
consists of a $\CCC$-profile $\und{c} = (c_1,\dots,c_n)$
equipped with a right permutation $\ast : \und{c} \to \mathrm{Rev}(\und{c}) = \und{c}\,\rho^{}_{\vert\und{c}\vert}$
satisfying $\ast \rho^{}_{\vert\und{c}\vert} \ast \rho^{}_{\vert\und{c}\vert} 
= e \in \Sigma_{\vert\und{c}\vert}$, where $e$ denotes the identity permutation.
In particular, any object $\und{c}\in\Sigma_\CCC$ carries
a canonical $\ast$-object structure given by $\rho^{}_{\vert\und{c}\vert} : 
\und{c} \to \und{c}\,\rho^{}_{\vert\und{c}\vert}$.
The assignment $\und{c} \mapsto (\rho^{}_{\vert\und{c}\vert} : 
\und{c} \to \und{c}\,\rho^{}_{\vert\und{c}\vert})$ 
defines a functor $ \rho : \Sigma_\CCC \to \astObj(\Sigma_\CCC,\mathrm{Rev},\id_{\ID_{\Sigma_\CCC}})$
that is a section of the forgetful functor $U :  
\astObj(\Sigma_\CCC,\mathrm{Rev},\id_{\ID_{\Sigma_\CCC}} ) \to\Sigma_\CCC $.
\end{ex}

For any involutive category $(\CC,J,j)$, there exists a 
forgetful functor $U : \astObj(\CC,J,j) \to \CC$ 
specified by $(\ast : c\to Jc) \mapsto c$.
If the category $\CC$ has coproducts, 
we can define for any object $c\in\CC$ a morphism
\begin{flalign}\label{eqn:freeastobject}
F(c) ~:=~\Big(\xymatrix{
c \sqcup Jc \,\cong \, Jc \sqcup c \,\ar[r]^-{\id \sqcup j_c} & \, Jc \sqcup J^2c \, \cong \, J(c \sqcup Jc) 
}\Big)
\end{flalign} 
in $\CC$, where in the last step we used that $J$ 
preserves coproducts because of Lemma \ref{lem:Jselfadjoint}. 
One can easily check that \eqref{eqn:freeastobject} defines a $\ast$-object in 
$(\CC,J,j)$, i.e.\ $F(c) \in\astObj(\CC,J,j)$. Another direct computation shows
\begin{propo}
Let $(\CC,J,j)$ be an involutive category that admits coproducts. 
The assignment $c \mapsto F(c)$ given by \eqref{eqn:freeastobject} 
naturally extends to a functor $F: \CC \to \astObj(\CC,J,j)$, which is 
a left adjoint of the forgetful functor $U : \astObj(\CC,J,j) \to \CC$.
\end{propo}
\begin{rem}\label{rem:astObjlim}
\cite[Lemma 5]{Jacobs} shows that $\astObj(\CC,J,j)$ inherits 
all limits and colimits that exist in $\CC$. These are preserved 
by the forgetful functor $U : \astObj(\CC,J,j) \to \CC$. 
\end{rem}

As noted in \cite[Lemma 6]{Jacobs}, the assignment 
of the categories of $\ast$-objects extends to a $2$-functor 
\begin{flalign}\label{eqn:astObj2fun}
\astObj\, :\,  \ICat ~\longrightarrow~ \Cat \quad.
\end{flalign}
Concretely, this $2$-functor is given by the following assignment: 
\begin{itemize}
\item an involutive category $(\CC,J,j)$ is mapped to its category of $\ast$-objects $\astObj(\CC,J,j)$;
\item an involutive functor $(F,\nu) : (\CC,J,j)\to (\CC^\prime,J^\prime,j^\prime)$
is mapped to the functor $\astObj(F,\nu) : \astObj(\CC,J,j)\to \astObj(\CC^\prime,J^\prime,j^\prime)$
that acts on objects as
\begin{flalign}
\astObj(F,\nu)\big(\ast :c\to Jc\big)  ~:= ~\big(\xymatrix@C=1.5em{Fc \ar[r]^-{F\ast} & FJc \ar[r]^-{\nu_c} & J^\prime Fc} \big)
\end{flalign}
and on morphisms as $F$;
\item an involutive natural transformation
$\zeta : (F,\nu) \to (G,\chi)$  is mapped to the natural transformation
$\astObj(\zeta) : \astObj(F,\nu) \to \astObj(G,\chi)$ with components
$\astObj(\zeta)_{(\ast : c\to Jc)}^{} :=  \zeta_c$,
for all $(\ast :c\to Jc)\in\astObj(\CC,J,j)$.
\end{itemize}
Recalling the trivial involutive categories from Example \ref{ex:trivial},
we obtain another $2$-functor
\begin{flalign}\label{eqn:triv}
\triv\, :\,  \Cat~ \longrightarrow ~ \ICat \quad.
\end{flalign}
Concretely, this $2$-functor assigns to a category $\CC$
the trivial involutive category $(\CC,\ID_\CC, \id_{\ID_\CC})$,
to a functor $F : \CC\to \CC^\prime$ the involutive functor
$(F,\id_F) : (\CC,\ID_\CC, \id_{\ID_\CC}) \to (\CC^\prime ,\ID_{\CC^\prime}, \id_{\ID_{\CC^\prime}})$, and to
a natural transformation $\zeta : F\to G$ the involutive natural transformation
$\zeta : (F,\id_F) \to (G,\id_G)$.
\begin{theo}\label{theo:trivastObjadjunction}
The $2$-functors \eqref{eqn:astObj2fun} and \eqref{eqn:triv}
form a $2$-adjunction 
\begin{flalign}\label{eqn:trivastObjadjunction}
\xymatrix{
\triv \,:\, \Cat~\ar@<0.5ex>[r]&\ar@<0.5ex>[l]  ~\ICat \,:\, \astObj\quad.
}
\end{flalign}
The unit $\eta : \ID_{\Cat} \to \astObj\,\triv$ and counit
$\epsilon : \triv \, \astObj\to \ID_{\ICat}$ $2$-natural transformations
are stated explicitly in the proof below.
\end{theo}
\begin{proof}
The component at $\CC\in\Cat$ of the $2$-natural transformation 
$\eta$ is the functor
\begin{flalign}
\eta_\CC \,:\, \CC ~\longrightarrow~ \astObj\big(\triv(\CC)\big)
\end{flalign}
that equips objects with their identity involution (cf.\ Example \ref{ex:trivial2}),
i.e.\ $c\mapsto (\id_c : c\to c)$.
The component at $(\CC,J,j)\in\ICat$ of the $2$-natural transformation $\epsilon$
is the involutive functor
\begin{flalign}
\epsilon_{(\CC,J,j)} = (U,\nu) \,:\, \triv\big(\astObj(\CC,J,j)\big)~\longrightarrow ~(\CC,J,j)\quad,
\end{flalign}
where $U : \astObj(\CC,J,j)\to \CC$ is the forgetful functor $(\ast :c\to Jc)\mapsto c$
and its involutive structure $\nu : U \to JU$ is the natural transformation defined by the components
$\nu_{(\ast:c\to Jc)} = \ast : c\to Jc$, for all $(\ast:c\to Jc) \in \astObj(\CC,J,j)$.
An elementary check shows that $\eta$ and $\epsilon$ are indeed $2$-natural transformations
that satisfy the triangle identities, hence \eqref{eqn:trivastObjadjunction} is a 
$2$-adjunction with unit $\eta$ and counit $\epsilon$.
\end{proof}

\begin{rem}\label{rem:astObjCart}
Notice that both $\Cat$ and $\ICat$ carry
a Cartesian monoidal structure, which is concretely given by the product categories 
$\CC\times \DD$ in $\Cat$ and the product involutive categories 
$(\CC,J,j)\times (\DD,K,k) = (\CC\times \DD,J\times K,j\times k)$
in $\ICat$. Because $\astObj$ is a right adjoint functor, it follows that there are canonical isomorphisms
\begin{flalign}
\astObj \big( (\CC,J,j) \times (\DD,K,k) \big) ~
\cong ~ \astObj(\CC,J,j) \times \astObj(\DD,K,k) \quad, 
\end{flalign}
for all involutive categories $(\CC,J,j)$ and $(\DD,K,k)$. 
\end{rem}

We conclude this section with a useful result 
that allows us to detect involutive categories 
carrying a trivial involutive structure. 
\begin{theo}\label{theo:iso2trivial}
Let $(\CC,J,j)$ be an involutive category. Any section
$\ast : \CC\to \astObj(\CC,J,j)$ of the forgetful functor $U : \astObj(\CC,J,j)\to \CC$
canonically determines an $\ICat$-isomorphism between $(\CC,J,j)$
and the trivial involutive category $(\CC,\ID_{\CC},\id_{\ID_\CC})$.
In particular, if a section of $U$ exists, then the involutive categories
$(\CC,J,j)$ and  $(\CC,\ID_{\CC},\id_{\ID_\CC})$ are isomorphic.
\end{theo}
\begin{proof}
A section $\ast : \CC\to \astObj(\CC,J,j)$ of $U$
assigns to each $c\in\CC$ a $\ast$-object $\ast_c^{} : c\to Jc$ 
and to each $\CC$-morphism $f :c\to c^\prime$ a $\ast$-morphism
\begin{flalign}\label{eqn:natast}
\xymatrix{
\ar[d]_-{\ast_c^{}}c \ar[rr]^-{f} && c^\prime\ar[d]^-{\ast_{c^\prime}^{}}\\
Jc \ar[rr]_-{Jf} && Jc^\prime
}
\end{flalign}
Notice that this diagram implies that $\ast_c $ are the components of a 
natural transformation $\ast : \ID_{\CC}\to J$. It is straightforward to check that
$(\ID_\CC,\ast) : (\CC,\ID_{\CC},\id_{\ID_\CC}) \to (\CC,J,j)$ is an involutive functor,
which is invertible via the involutive functor $(\ID_\CC,\ast^{-1}) :  (\CC,J,j)\to (\CC,\ID_{\CC},\id_{\ID_\CC}) $.
\end{proof}

\begin{cor}
The involutive category $(\Sigma_\CCC,\mathrm{Rev},\id_{\ID_{\Sigma_\CCC}})$ 
of $\CCC$-profiles equipped with reversal as involutive structure
(cf.\ Examples \ref{ex:Cprofiles} and \ref{ex:Cprofiles2})
is isomorphic to the trivial involutive category $(\Sigma_\CCC,\ID_{\Sigma_\CCC},\id_{\ID_{\Sigma_\CCC}})$.
\end{cor}


\section{\label{sec:invmoncat}Involutive structures on monoidal categories}
In this section we review involutive (symmetric) monoidal categories 
and $\ast$-monoids therein. We again shall follow mostly the definitions 
and conventions of Jacobs \cite{Jacobs}. Our main goal is to
clarify and work out the $2$-functorial behavior of the assignment
of the categories of $\ast$-objects and monoids to involutive (symmetric) monoidal categories.
To fix our notations, we start with a brief review of some basic aspects 
of (symmetric) monoidal categories and monoids therein.

\subsection{(Symmetric) monoidal categories and monoids}
Recall that a {\em monoidal category} $(\CC,\otimes,I,\alpha,\lambda,\rho)$ consists of
a category $\CC$, a functor $\otimes : \CC\times\CC\to \CC$, an object $I\in \CC$ 
and three natural isomorphisms
\begin{subequations}
\begin{flalign}
\alpha \,:\,  \otimes\,(\otimes\times \ID_{\CC})~&\longrightarrow~\otimes\, (\ID_\CC\times \otimes)\quad,\\
\lambda \,:\, I\otimes (-) ~&\longrightarrow~\ID_{\CC}\quad,\\ 
\rho\,:\, (-)\otimes I ~&\longrightarrow~\ID_{\CC}\quad,
\end{flalign}
\end{subequations}
which satisfy the pentagon and triangle identities. 
We follow the usual abuse of notation
and often denote a monoidal category by its underlying category $\CC$. 
The associator $\alpha$ and the unitors $\lambda$ and $\rho$ will always be suppressed. 
Given two monoidal categories $\CC$ and $\CC^\prime$, a  {\em (lax) monoidal functor} from $\CC$ to $\CC^\prime$
is a triple $(F,F_2,F_0)$ consisting of a functor $F : \CC\to\CC^\prime$, a natural transformation
\begin{subequations}
\begin{flalign}
F_2 \,:\, \otimes^\prime\,(F\times F)~\longrightarrow~F\,\otimes \quad,
\end{flalign}
and a $\CC^\prime$-morphism
\begin{flalign}
F_0 \,:\, I^\prime ~\longrightarrow~F I\quad,
\end{flalign}
\end{subequations}
which are required to satisfy the usual coherence conditions involving the associators and unitors.
We often denote a monoidal functor by its underlying functor $F : \CC\to \CC^\prime$.
A {\em monoidal natural transformation} $\zeta : F \to G$ between 
monoidal functors $F=(F,F_2,F_0)$ and $G=(G,G_2,G_0)$ is a natural transformation
$\zeta : F\to G$ satisfying
\begin{flalign}
\xymatrix{
\ar[d]_-{F_2}\otimes^\prime\,(F\times F)\ar[rr]^-{\otimes^\prime\,(\zeta\times\zeta)} && \otimes^\prime\,(G\times G)\ar[d]^-{G_2}&& &\ar[dl]_-{F_0}I^\prime\ar[dr]^-{G_0}&\\
F\,\otimes \ar[rr]_-{\zeta\,\otimes}&& G\,\otimes && F I \ar[rr]_-{\zeta_I^{}}&&G I
}
\end{flalign}
\begin{propo}
Monoidal categories, (lax) monoidal functors and monoidal natural transformations
form a $2$-category $\MonCat$.
\end{propo}

A {\em symmetric monoidal category} is a monoidal category $\CC$ together with
a natural isomorphism called \emph{braiding}
\begin{flalign}
\tau\,:\, \otimes ~\longrightarrow~\otimes^{\op} := \otimes\,\sigma\quad
\end{flalign}
from the tensor product to the opposite tensor product, where $\sigma : \CC\times\CC\to\CC\times\CC$ is 
the flip functor $(c_1,c_2)\mapsto (c_2,c_1)$, which
satisfies the hexagon identities and the symmetry constraint
\begin{flalign}
\xymatrix{
\ar[dr]_-{\tau}\otimes \ar[rr]^-{\id_\otimes} && \otimes =\otimes\,\sigma^2\\
&\otimes\,\sigma \ar[ru]_-{\tau\,\sigma}&
}
\end{flalign}
We often denote a symmetric monoidal category by its underling category $\CC$. 
A {\em symmetric monoidal functor}
is a monoidal functor $F : \CC\to \CC^\prime$  that preserves the braidings, i.e.\
\begin{flalign}
\xymatrix{
\ar[d]_-{F_2} \otimes^\prime\,(F\times F)\ar[rr]^-{\tau^\prime \,(F\times F)} &&  \otimes^{\prime}\,\sigma (F\times F)=\otimes^{\prime}\,(F\times F)\sigma \ar[d]^-{F_2\sigma}\\
F\otimes \ar[rr]_-{F\tau}&&F\otimes \sigma
}
\end{flalign}
commutes. A {\em symmetric monoidal natural transformation} is just a monoidal natural transformation
between symmetric monoidal functors.
\begin{propo}
Symmetric monoidal categories, symmetric monoidal functors and symmetric 
monoidal natural transformations form a $2$-category $\SMonCat$.
\end{propo}

\begin{defi}\label{def:monoid}
A {\em monoid} in a (symmetric) monoidal category $\CC$
is a triple $(M,\mu,\eta)$ consisting of an object $M\in\CC$
and two $\CC$-morphisms $\mu : M\otimes M\to M$ (called {\em multiplication})
and $\eta : I\to M$ (called {\em unit}) satisfying the associativity and unitality axioms.
A {\em monoid morphism} $f : (M,\mu,\eta)\to (M^\prime,\mu^\prime,\eta^\prime)$
is a $\CC$-morphism $f:M\to M^\prime$ preserving multiplications and units.
We denote the category of monoids in $\CC$ by $\Mon(\CC)$.
\end{defi}

The assignment of the categories of monoids extends to a $2$-functor
\begin{flalign}\label{eqn:Mon2fun}
\Mon \,:\, \mathbf{(S)MCat} ~\longrightarrow~ \Cat \quad. 
\end{flalign}
Concretely, this $2$-functor is given by the following assignment: 
\begin{itemize}
\item a (symmetric) monoidal category $\CC$ is mapped to its category of monoids $\Mon(\CC)$;
\item a (symmetric) monoidal functor $F : \CC\to \CC^\prime$ is mapped to 
the functor $\Mon(F) : \Mon(\CC)\to \Mon(\CC^\prime)$ that acts on objects as
\begin{flalign}
\Mon(F)\big(M,\mu,\eta\big) \,:= \,\Big(FM, \xymatrix@C=2.1em{FM\otimes^\prime FM \ar[r]^-{{F_2}_{M,M}}& F(M\otimes M)
\ar[r]^-{F\mu} & FM}\!\!,\xymatrix@C=1.1em{I^\prime \ar[r]^-{F_0}& FI\ar[r]^-{F\eta}& FM}\!\Big)  
\end{flalign}
and on morphisms as $F$;
\item  a (symmetric) monoidal natural transformation
$\zeta : F\to G$ is mapped to the natural transformation $\Mon(\zeta) : \Mon(F)\to \Mon(G)$ 
with components $\Mon(\zeta)_{(M,\mu,\eta)} := \zeta_M$, for all $(M,\mu,\eta)\in\Mon(\CC)$.
\end{itemize}

\subsection{\label{sec:ISMCat}Involutive (symmetric) monoidal categories}
The following definition of an involutive (symmetric) monoidal category 
is due to \cite{Jacobs}. We prefer this definition 
over the one in \cite{Egger, BeggsMajid} as 
it has the advantage that the category of $\ast$-objects inherits a monoidal structure 
(cf.\ \cite[Proposition 1]{Jacobs} and Proposition \ref{propo:astObjmonoidal} in the present paper). 
This has interesting consequences for the theory of involutive monads in \cite{Jacobs}
and the developments in our present paper.

\begin{defi}\label{def:invmoncat}
An {\em involutive (symmetric) monoidal category} is a triple $(\CC,J,j)$ consisting of
a (symmetric) monoidal category $\CC$,
a (symmetric) monoidal endofunctor $J = (J,J_2,J_0) : \CC\to \CC$ and a (symmetric) 
monoidal natural isomorphism $j : \ID_{\CC}\to J^2$ satisfying 
\begin{flalign}
jJ \,=\, Jj \,:\, J~\longrightarrow~ J^3\quad.
\end{flalign}
\end{defi}

The following statement is proven in \cite[Lemma 7]{Jacobs}.
\begin{lem}\label{lem:Jstrong}
For any involutive (symmetric) monoidal category,
the (symmetric) monoidal endofunctor  $J = (J,J_2,J_0) : \CC\to\CC$ is {\em strong},
i.e.\ $J_2 : \otimes\, (J\times J) \to J\,\otimes$ and $J_0 : I\to JI$ are isomorphisms.
\end{lem}

\begin{rem}
Let us emphasize again and more clearly that our Definition \ref{def:invmoncat} of 
involutive (symmetric) monoidal categories
agrees with the one of Jacobs \cite{Jacobs}. The definitions in \cite{BeggsMajid} and
\cite{Egger} are different because their analog of $J_2$ is order-reversing, i.e.\ a natural 
isomorphism $\otimes^\op\,(J\times J) \to J\,\otimes$. The reason why we consider order-preserving $J_2$ as in
\cite{Jacobs} is that this is better suited for our development of involutive operad theory, 
cf.\ Remark \ref{rem:whynonrev} below.
\end{rem}

\begin{rem}\label{rem:invmoncat1}
The condition for $j : \ID_{\CC}\to J^2$ to be a (symmetric) monoidal natural transformation
explicitly means that the diagrams
\begin{flalign}\label{eqn:jmonoidal}
\xymatrix{
\ar[dd]_-{\id_{\otimes}} \otimes \ar[rr]^-{\otimes\,(j\times j)} && {\otimes} (J^2\times J^2) \ar[d]^-{J_2 (J\times J)} &&&\ar[dr]^-{J_0} I \ar[dl]_-{\id_I}&\\
 && J {\otimes} (J\times J)\ar[d]^-{J\,J_2}&& \ar[d]_-{\id_I} I &&JI\ar[d]^-{J J_0}\\
\otimes \ar[rr]_-{j\,\otimes} &&J^2{\otimes} && I \ar[rr]_-{j_I}&&J^2 I
}
\end{flalign}
commute. One may reinterpret these diagrams as follows: 
The left diagram states that $(\otimes, J_2): (\CC,J,j) \times (\CC,J,j) \to (\CC,J,j)$ 
is an involutive functor on the product involutive category
$(\CC,J,j)\times (\CC,J,j)=(\CC\times \CC,J\times J,j\times j)$,
see also Remark \ref{rem:astObjCart}.
The right diagram states that $(J_0 : I\to JI ) \in\astObj(\CC,J,j)$
is a $\ast$-object in $(\CC,J,j)$. These two structures allow us to 
endow the functor $I \otimes (-) : \CC\to\CC$ with an involutive structure 
$I \otimes J(-) \to J(I\otimes (-))$ defined by the components
\begin{flalign}
\xymatrix@C=3em{I\otimes Jc \ar[r]^-{J_0\otimes \id} & JI\otimes Jc\ar[r]^-{{J_2}_{I,c}} & J(I\otimes c)}\quad,
\end{flalign}
for all $c\in\CC$. An analogous statement holds true for the functor $(-)\otimes I : \CC\to \CC$.
The axioms for the (symmetric) monoidal structure on $J$ can then be reinterpreted as the 
equivalent property that the associator and unitors (as well as the braiding in the symmetric case)
are involutive natural transformations.
\sk

Summing up, we obtain an equivalent description of an involutive (symmetric) monoidal category 
in terms of the following data: 
An involutive category $(\CC,J,j)$, 
an involutive functor $(\otimes,J_2) : (\CC,J,j)\times (\CC,J,j)\to (\CC,J,j)$, 
a $\ast$-object $(J_0 : I\to JI)\in\astObj(\CC,J,j)$ and involutive natural transformations 
for the associator and unitors (as well as the braiding in the symmetric case), 
which satisfy analogous axioms as those for (symmetric) monoidal categories.
This alternative point of view is useful for \eqref{eqn:astObjmonoidal} and \eqref{eqn:astObjunit} below. 
\end{rem}

\begin{ex}\label{ex:trivial3}
For any (symmetric) monoidal category $\CC$, the triple $(\CC,\ID_\CC,\id_{\ID_{\CC}})$,
with $\ID_\CC$ the identity (symmetric) monoidal functor and $\id_{\ID_\CC}$ the identity (symmetric)
monoidal natural transformation, defines an involutive (symmetric) monoidal category.
We call this the {\em trivial involutive (symmetric) monoidal category} over $\CC$.
\end{ex}

\begin{ex}\label{ex:vec3}
Let us equip the category of complex vector spaces $\Vec_\bbC$  with its
standard symmetric monoidal structure where $\otimes$ is the usual tensor product,
$I=\bbC$ is the ground field and $\tau$ is given by the flip maps $\tau_{V,W}^{}:V\otimes W \to W\otimes V
\,,~v \otimes w\mapsto w\otimes v$. The endofunctor 
$\overline{(-)} : \Vec_{\bbC}\to \Vec_{\bbC}$ from Example \ref{ex:vec} can be promoted to a symmetric
monoidal functor by using the canonical maps
${\overline{(-)}_2}_{V,W} :\overline{V}\otimes \overline{W}\to \overline{V\otimes W}$
and complex conjugation $\overline{(-)}_0 : \bbC\to\overline{\bbC}$.
The resulting triple $(\Vec_\bbC,\overline{(-)},\id_{\ID_{\Vec_\bbC}})$ is an involutive symmetric monoidal category.
\end{ex}

\begin{ex}\label{ex:Cprofiles3}
Recall the groupoid of $\CCC$-profiles $\Sigma_\CCC$ from Example \ref{ex:Cprofiles}.
The category $\Sigma_\CCC$ may be equipped with the symmetric monoidal structure given by
concatenation of $\CCC$-profiles, i.e.\ $\und{c}\otimes\und{d} = (c_1,\dots,c_n,d_1,\dots,d_m)$,
$I=\emptyset$ is the empty $\CCC$-profile 
and $\tau_{\und{c},\und{d}}:=\tau\langle\vert\und{c}\vert,\vert\und{d}\vert\rangle: 
\und{c}\otimes\und{d}\to \und{d}\otimes\und{c}$ is the block transposition.
The reversal endofunctor $\mathrm{Rev} : \Sigma_\CCC\to\Sigma_\CCC$ can be promoted to a symmetric
monoidal functor by using 
\begin{flalign}
{\mathrm{Rev}_2}_{\und{c},\und{d}} \,:= \, \tau\langle \vert\und{c}\vert,\vert\und{d}\vert\rangle ~:~
\mathrm{Rev}(\und{c})\otimes\mathrm{Rev}(\und{d})~\longrightarrow  ~\mathrm{Rev}(\und{c}\otimes\und{d})\quad
\end{flalign}
and $\mathrm{Rev}_0:= \id_\emptyset^{} : \emptyset \to \mathrm{Rev}(\emptyset)=\emptyset$.
The resulting triple $(\Sigma_\CCC,\mathrm{Rev},\id_{\ID_{\Sigma_\CCC}})$ is an involutive symmetric
monoidal category.
\end{ex}

\begin{defi}\label{def:invmonfunandnat}
An {\em involutive (symmetric) monoidal functor} $(F,\nu) : (\CC,J,j) \to (\CC^\prime,J^\prime,j^\prime)$
consists of a (symmetric) monoidal functor $F=(F,F_2,F_0) : \CC\to\CC^\prime$ 
and a (symmetric) monoidal natural transformation $\nu : FJ \to J^\prime F$
satisfying the analog of diagram \eqref{eqn:invfundiagram} in Definition \ref{def:invfunandnat}.

An {\em involutive (symmetric) monoidal natural transformation} $\zeta : (F,\nu)\to (G,\chi)$
between involutive (symmetric) monoidal functors $(F,\nu),(G,\chi): (\CC,J,j) \to (\CC^\prime,J^\prime,j^\prime) $
is a natural transformation $\zeta : F\to G$ that is both involutive and (symmetric) monoidal. 
\end{defi}

\begin{propo}
Involutive (symmetric) monoidal categories, involutive (symmetric)
monoidal functors and involutive (symmetric) monoidal 
natural transformations form a $2$-category $\mathbf{I(S)MCat}$. 
\end{propo}

\begin{rem}\label{rem:invmoncat2}
The condition for the natural transformation $\nu: FJ\to J^\prime F$
to be monoidal explicitly means that the diagrams 
\begin{flalign}
\xymatrix{
\otimes^\prime(FJ \times FJ) 
\ar[rr]^-{\otimes^\prime (\nu \times \nu)} \ar[d]_-{F_2(J \times J)}
& & \otimes^\prime (J^\prime F \times J^\prime F) 
\ar[d]^-{J^\prime_2(F \times F)} 
&& & I^\prime \ar[ld]_-{F_0} \ar[rd]^-{J^\prime_0} \\
F \otimes (J \times J) \ar[d]_-{FJ_2} 
& & J^\prime \otimes^\prime (F \times F) \ar[d]^-{J^\prime F_2} 
&& FI \ar[d]_-{FJ_0} & & J^\prime I^\prime \ar[d]^-{J^\prime F_0} \\
FJ\otimes \ar[rr]_-{\nu \otimes} & & J^\prime F\otimes
&& FJI \ar[rr]_-{\nu_I} & & J^\prime FI
}
\end{flalign}
commute. From the perspective established in 
Remark \ref{rem:invmoncat1}, one may reinterpret these
diagrams as follows: The left diagram states that $F_2$ is
an involutive natural transformation 
\begin{flalign}\label{eqnF2tmprem}
F_2\, :\,  (\otimes^\prime, J^\prime_2)\; \big( (F, \nu) \times (F, \nu) \big)
~\longrightarrow~ (F, \nu)\; (\otimes, J_2) 
\end{flalign}
of involutive functors from $(\CC,J,j)\times(\CC,J,j) $ to $(\CC^\prime, J^\prime, j^\prime)$.
The right diagram states that $F_0$ defines a morphism
\begin{flalign}\label{eqnF0tmprem}
F_0\, :\, \big(J^\prime_0 : I^\prime \to J^\prime I^\prime\big)~\longrightarrow~ \astObj(F,\nu) \big(J_0 : I\to JI \big)
\end{flalign}
in the category $\astObj(\CC^\prime,J^\prime,j^\prime)$ of $\ast$-objects in
$(\CC^\prime,J^\prime,j^\prime)$. 
\sk

Summing up, we obtain an equivalent description of an involutive (symmetric) monoidal functor 
in terms of the following data:  An involutive functor $(F,\nu) : (\CC,J,j)\to (\CC^\prime,J^\prime,j)$, 
an involutive natural transformation $F_2$ as in \eqref{eqnF2tmprem}
and a $\ast$-morphism $F_0$ as in \eqref{eqnF0tmprem}, 
which satisfy axioms analogous to those for a (symmetric) monoidal functor.
This alternative point of view is useful for \eqref{eqn:astObjmonoidal2fun1cell} below. 
\end{rem}

\begin{rem}
Let us summarize Remarks \ref{rem:invmoncat1} and \ref{rem:invmoncat2}
by one slogan: Involutive (symmetric) monoidal categories are the same things as
(symmetric) monoidal involutive categories.
\end{rem}

Let $(\CC,J,j)$ be an involutive (symmetric) monoidal category and consider its category
of $\ast$-objects $\astObj(\CC,J,j)$. Making use of the $2$-functor
$\astObj: \ICat \to \Cat$ given in  \eqref{eqn:astObj2fun}, we may equip the category $\astObj(\CC,J,j)$
with a (symmetric) monoidal structure. Concretely, the tensor product
functor is given by
\begin{flalign}\label{eqn:astObjmonoidal}
\xymatrix{
\ar[d]_-{\cong} \astObj(\CC,J,j) \times \astObj(\CC,J,j) \ar[rr]^-{\otimes } && \astObj(\CC,J,j)\\
\astObj\big((\CC,J,j)\times(\CC,J,j)\big) \ar[rru]_-{~~~~~~\astObj(\otimes,J_2)} &&
}
\end{flalign}
where the vertical isomorphism was explained in Remark \ref{rem:astObjCart} 
and the involutive functor $(\otimes,J_2)$ in Remark \ref{rem:invmoncat1}.
The unit object
\begin{flalign}\label{eqn:astObjunit}
\big(J_0 : I\to JI\big) \in \astObj(\CC,J,j)
\end{flalign}
is the $\ast$-object constructed in Remark \ref{rem:invmoncat1}.
The associator and unitors (as well as the braiding in the symmetric case)
are obtained by applying the $2$-functor $\astObj$ to the associator and unitors 
(as well as the braiding in the symmetric case) of $(\CC,J,j)$,
which makes sense because Remark \ref{rem:invmoncat1} 
shows that these are involutive natural transformations.
Let us also mention that the
tensor product of two $\ast$-objects $(\ast : c\to Jc), (\ast^\prime : 
c^\prime\to Jc^\prime)\in\astObj(\CC,J,j)$ explicitly reads as
\begin{flalign}
( \ast : c\to Jc) \otimes (\ast^\prime:c^\prime \to Jc^\prime) \,=\,
\Big( \xymatrix{ c\otimes c^\prime \ar[r]^-{\ast\otimes \ast^\prime}
& Jc\otimes Jc^\prime \ar[r]^-{{J_{2}}_{c,c^\prime}} & J(c\otimes c^\prime) }
\Big)\quad.
\end{flalign}
Summing up, we have proven
\begin{propo}\label{propo:astObjmonoidal}
Let $(\CC,J,j)$ be an involutive (symmetric) monoidal category. Then the category of $\ast$-objects 
$\astObj(\CC,J,j)$ is a (symmetric) monoidal category with 
tensor product \eqref{eqn:astObjmonoidal} 
and unit object \eqref{eqn:astObjunit}. Moreover, if
$(\CC,J,j)$ is also closed, i.e.\ it has internal homs, then 
$\astObj(\CC,J,j)$ is closed too (cf.\ \cite[Proposition 1]{Jacobs}). 
\end{propo}

The assignment of the (symmetric) monoidal categories of
$\ast$-objects extends to a $2$-functor
\begin{flalign}\label{eqn:astObjmonoidal2fun}
\astObj \,:\, \mathbf{I(S)MCat} ~\longrightarrow ~\mathbf{(S)MCat}\quad,
\end{flalign}
which we shall denote with an abuse of notation by the same symbol as the
$2$-functor in \eqref{eqn:astObj2fun}.
Concretely, this $2$-functor is given by the following assignment: 
\begin{itemize}
\item an involutive (symmetric) monoidal category $(\CC,J,j)$ is mapped to
the (symmetric) monoidal category $\astObj(\CC,J,j)$ given in Proposition \ref{propo:astObjmonoidal};
\item an involutive (symmetric) monoidal functor $(F,\nu) : (\CC,J,j)\to (\CC^\prime,J^\prime,j^\prime)$
is mapped to the (symmetric) monoidal functor
\begin{subequations}\label{eqn:astObjmonoidal2fun1cell}
\begin{flalign}
\astObj(F,\nu) \,:\, \astObj(\CC,J,j)~ \longrightarrow ~\astObj(\CC^\prime,J^\prime,j^\prime)
\end{flalign}
with underlying functor as in \eqref{eqn:astObj2fun} and (symmetric) monoidal structure
given by
\begin{flalign}
\astObj(F)_2 \,:=\, \astObj(F_2) \quad, \qquad \astObj(F)_0 \,:= \,F_0 \quad,
\end{flalign}
\end{subequations}
where $F_2$ and $F_0$ should be interpreted according to Remark \ref{rem:invmoncat2};
\item an involutive (symmetric) monoidal natural transformation 
$\zeta: (F,\nu) \to (G,\chi)$ is mapped to the (symmetric) monoidal natural transformation 
determined by \eqref{eqn:astObj2fun}.
\end{itemize}

\begin{rem}\label{rem:astObjlift}
Notice that the $2$-functor $\astObj: \mathbf{I(S)MCat} \to \mathbf{(S)MCat}$ 
given in \eqref{eqn:astObjmonoidal2fun} is a lift of the $2$-functor
$\astObj: \ICat \to \Cat$ given in \eqref{eqn:astObj2fun} 
along the forgetful $2$-functors $\mathrm{forget_{\otimes}} : \mathbf{I(S)MCat} \to \ICat$ 
and $\mathrm{forget_{\otimes}} : \mathbf{(S)MCat} \to \Cat$ that forget the (symmetric) 
monoidal structures. More precisely, using the explicit descriptions
of our $2$-functors, one easily confirms that the diagram
\begin{flalign}
\xymatrix{
\mathbf{I(S)MCat} \ar[rr]^-{\astObj} \ar[d]_-{\mathrm{forget}_\otimes} 
& & \mathbf{(S)MCat} \ar[d]^-{\mathrm{forget}_\otimes} \\
\ICat \ar[rr]_-{\astObj} & & \Cat
}
\end{flalign}
of $2$-categories and $2$-functors commutes (on the nose). 
\end{rem}

We conclude this section with a useful result that 
generalizes Theorem \ref{theo:iso2trivial} to the 
(symmetric) monoidal setting. 
Let us first notice that the forgetful functor 
$U : \astObj(\CC,J,j)\to \CC$ satisfies
$\otimes (U\times U) = U\otimes$ and 
$U(J_0 : I \to JI) = I$, hence it
can be promoted to a (symmetric) monoidal functor
via the trivial (symmetric) monoidal structure
$U_2 = \id_{U\otimes}$ and $U_0 = \id_I$.
\begin{theo}\label{theo:iso2trivialmonoidal}
Let $(\CC,J,j)$ be an involutive (symmetric) monoidal category. Any
(symmetric) monoidal section $\ast  : \CC\to \astObj(\CC,J,j)$ of the forgetful (symmetric) monoidal 
functor $U : \astObj(\CC,J,j)\to \CC$ canonically determines an $\mathbf{I(S)MCat}$-isomorphism 
between $(\CC,J,j)$ and the trivial involutive (symmetric) monoidal  category $(\CC,\ID_{\CC},\id_{\ID_\CC})$.
In particular, if such a section of $U$ exists, then the involutive (symmetric) monoidal categories
$(\CC,J,j)$ and  $(\CC,\ID_{\CC},\id_{\ID_\CC})$ are isomorphic.
\end{theo}
\begin{proof}
Using that the (symmetric) monoidal structure on $U$ is trivial, i.e.\ 
$U_2 = \id_{U\otimes}$ and $U_0 = \id_I$, and also that $U$ is a faithful functor,
one observes that the (symmetric) monoidal structure on the (symmetric)
monoidal section $\ast  : \CC\to \astObj(\CC,J,j)$ is necessarily trivial.
The proof then proceeds analogously to the one of Theorem \ref{theo:iso2trivial}.
\end{proof}

\begin{cor}\label{cor:SigmaCCCtrivial}
The involutive symmetric monoidal category $(\Sigma_\CCC,\mathrm{Rev},\id_{\ID_{\Sigma_\CCC}})$ 
of $\CCC$-profiles equipped with reversal as involutive structure 
(cf.\ Example \ref{ex:Cprofiles3}) is isomorphic to the trivial involutive symmetric monoidal
category $(\Sigma_\CCC,\ID_{\Sigma_\CCC},\id_{\ID_{\Sigma_\CCC}})$.
\end{cor}
\begin{proof}
By Theorem \ref{theo:iso2trivialmonoidal}, it is sufficient to construct 
a symmetric monoidal section $\rho = (\rho,\rho_2,\rho_0) : \Sigma_\CCC \to 
\astObj(\Sigma_\CCC,\mathrm{Rev},\id_{\ID_{\Sigma_\CCC}})$ of the forgetful 
symmetric monoidal functor $U$. Taking the underlying functor
as in Example \ref{ex:Cprofiles2}, i.e.\ 
$\rho : \und{c}\mapsto (\rho_{\vert\und{c}\vert} : \und{c} \to \und{c} \rho_{\vert\und{c}\vert})$
with the order-reversal permutations $\rho_{\vert\und{c}\vert}\in\Sigma_{\vert\und{c}\vert}$,
one easily checks that $\otimes (\rho\times \rho) = \rho\otimes$ and 
$\rho(\emptyset) = (\id_{\emptyset} : \emptyset \to \emptyset) =(\mathrm{Rev}_0 : \emptyset \to \mathrm{Rev}(\emptyset)) $.
We choose the trivial symmetric monoidal structure $\rho_2 = \id_{\rho\otimes}$ 
and $\rho_0 =\id_\emptyset$.
\end{proof}

\subsection{$\ast$-monoids}
Let us recall the $2$-functors 
$\Mon : \mathbf{(S)MCat} \to \Cat$ given in \eqref{eqn:Mon2fun},
$\astObj : \ICat\to\Cat$ given in \eqref{eqn:astObj2fun}
and its lift  $\astObj : \mathbf{I(S)MCat} \to \mathbf{S(M)Cat}$ 
given in \eqref{eqn:astObjmonoidal2fun}. The aim of this subsection
is to describe a $2$-functor $\Mon : \mathbf{I(S)MCat} \to \ICat$
that lifts $\Mon : \mathbf{(S)MCat} \to \Cat$ to the involutive setting,
such that the diagram
\begin{flalign}\label{eqn:astObjMoncommute}
\xymatrix{
\mathbf{I(S)MCat} \ar[rr]^-{\astObj} \ar@{-->}[d]_-{\Mon} & & \mathbf{(S)MCat} \ar[d]^-{\Mon} \\
\ICat \ar[rr]_-{\astObj} & & \Cat
}
\end{flalign}
of $2$-categories and $2$-functors commutes (on the nose). We then define $\ast$-monoids
in terms of the diagonal $2$-functor $\astMon : \mathbf{I(S)MCat} \to \Cat $ in this square.
\sk

Let us start with describing the $2$-functor
\begin{flalign}\label{eqn:IMon2fun}
\Mon \,:\, \mathbf{I(S)MCat} ~\longrightarrow~ \ICat
\end{flalign}
that lifts \eqref{eqn:Mon2fun} to the involutive setting in some detail:
\begin{itemize}
\item an involutive (symmetric) monoidal category $(\CC,J,j)$ is mapped to 
the involutive category
\begin{flalign}\label{eqn:IMonCCJj}
\Mon(\CC,J,j) \,:=\,\big(\Mon(\CC),\Mon(J),\Mon(j)\big)\in\ICat
\end{flalign}
given by evaluating the $2$-functor \eqref{eqn:Mon2fun}
on the (symmetric) monoidal category $\CC$, on the (symmetric) monoidal
endofunctor $J :\CC\to\CC$ and on the (symmetric) 
monoidal natural isomorphism $j:\ID_\CC\to J^2$;
\item an involutive (symmetric) monoidal functor
$(F,\nu) : (\CC,J,j)\to (\CC^\prime,J^\prime,j^\prime)$
is mapped to the involutive functor
\begin{flalign}\label{eqn:IMonFnu}
\Mon(F,\nu):= \big(\Mon(F),\Mon(\nu)\big)\,:\, \Mon(\CC,J,j)~\longrightarrow~
\Mon(\CC^\prime,J^\prime,j^\prime)
\end{flalign}
given by evaluating the $2$-functor \eqref{eqn:Mon2fun} on the (symmetric)
monoidal functor $F : \CC\to\CC^\prime$ and on the (symmetric)
monoidal natural transformation $\nu : FJ \to J^\prime F$;
\item an involutive (symmetric) monoidal natural transformation
$\zeta : (F,\nu) \to (G,\chi)$ is mapped to the involutive natural
transformation
\begin{flalign}
\Mon(\zeta) \,:\, \Mon(F,\nu)~\longrightarrow~ \Mon(G,\chi)
\end{flalign}
given by evaluating the $2$-functor \eqref{eqn:Mon2fun} on $\zeta$.
\end{itemize}
\begin{lem}\label{lem:astMoncommutes}
The diagram \eqref{eqn:astObjMoncommute} of $2$-categories and $2$-functors
commutes (on the nose).
\end{lem}
\begin{proof}
This is an elementary check using the explicit definitions of the $2$-functors
given in \eqref{eqn:Mon2fun}, \eqref{eqn:astObj2fun}, \eqref{eqn:astObjmonoidal2fun}
and \eqref{eqn:IMon2fun}. 
\end{proof}

\begin{defi}\label{def:astmonoids}
The $2$-functor $\astMon : \mathbf{I(S)MCat} \to \Cat$ is defined as the diagonal
$2$-functor in the commutative square \eqref{eqn:astObjMoncommute}, i.e.\
\begin{flalign}\label{eqn:astmonoids}
\xymatrix{
\ar[rrd]^-{~~\astMon}\mathbf{I(S)MCat} \ar[rr]^-{\astObj} \ar[d]_-{\Mon} & & \mathbf{(S)MCat} \ar[d]^-{\Mon} \\
\ICat \ar[rr]_-{\astObj} & & \Cat
}
\end{flalign}
For an involutive (symmetric) monoidal category $(\CC,J,j)$, we call
$\astMon(\CC,J,j)$ the category of {\em $\ast$-monoids} in $(\CC,J,j)$.
\end{defi}

\begin{rem}\label{rem:astmonoids}
Let $(\CC,J,j)$ be an involutive (symmetric) monoidal category.
We provide an explicit description of the objects and morphisms
in the associated category of $\ast$-monoids $\astMon(\CC,J,j)$,
which we shall call {\em $\ast$-monoids} and {\em $\ast$-monoid morphisms}.
Unpacking Definition \ref{def:astmonoids}, one obtains that a
$\ast$-monoid is a quadruple $(M,\mu,\eta,\ast)\in \astMon(\CC,J,j)$ 
consisting of an object $M\in\CC$ and three $\CC$-morphisms
$\mu : M\otimes M\to M$, $\eta : I\to M$ and $\ast : M\to JM$, which satisfy the 
following conditions:
\begin{itemize}
\item[(1)] $(M,\mu,\eta)$ is a monoid in the (symmetric) monoidal category $\CC$;
\item[(2)] $\ast : M\to JM$ is a $\ast$-object in the involutive category $(\CC,J,j)$;
\item[(3)] these two structures are compatible in the sense that the diagrams
\begin{flalign}\label{eqn:astmonoidcompatibilities}
\xymatrix@C=3.5em{
\ar[d]_-{J_0}I\ar[r]^-{\eta} & M\ar[d]^-{\ast} & \ar[d]_-{\mu}M\otimes M\ar[r]^-{\ast\otimes\ast} & JM\otimes JM \ar[r]^-{{J_2}_{M,M}} & J(M\otimes M)\ar[d]^-{J\mu}\\ 
J I \ar[r]_-{J\eta} & JM & M \ar[rr]_-{\ast}& &JM
}
\end{flalign}
in $\CC$ commute.
\end{itemize}
As a consequence of Lemma \ref{lem:astMoncommutes}, these conditions
have two equivalent interpretations which correspond to the counterclockwise and
clockwise paths in the commutative diagram \eqref{eqn:astmonoids}:
The first option is to regard $\ast : (M,\mu,\eta)\to \Mon(J)(M,\mu,\eta)$
as a $\ast$-object in the involutive category $\Mon(\CC,J,j)\in\ICat$. The second
option is to regard $\eta:(J_0: I\to JI) \to (\ast : M\to JM)$ and 
$\mu: (\ast:M\to JM) \otimes (\ast:M\to JM) \to (\ast : M\to JM)$ 
as the structure maps of a monoid in the (symmetric) monoidal category
$\astObj(\CC,J,j) \in \mathbf{(S)MCat}$.
\sk

A $\ast$-monoid morphism $f : (M,\mu,\eta,\ast) \to (M^\prime,\mu^\prime,\eta^\prime,\ast^\prime)$
is a $\CC$-morphism $f:M\to M^\prime$ that preserves both the monoid structures and $\ast$-involutions.
\end{rem}

\begin{ex}\label{ex:Vecexplicitmonoids}
Let us consider a $\ast$-monoid $(A,\mu,\eta,\ast)$
in the involutive symmetric monoidal category 
$(\Vec_\bbC,\ovr{(-)},\id_{\ID_{\Vec_\bbC}})$ from Example \ref{ex:vec3}.
In particular, the triple $(A,\mu,\eta)$ is an associative and unital algebra 
over $\bbC$ with multiplication $a\,b = \mu(a\otimes b)$ and unit $\oone=\eta(1)$.
By Example \ref{ex:vec2}, $\ast$ is a complex anti-linear automorphism
of $A$ that squares to the identity, i.e.\ $a^{\ast\ast} =a$. The compatibility conditions
in \eqref{eqn:astmonoidcompatibilities} state that
$\oone^\ast =\oone$ and $(a\,b)^\ast = a^\ast\,b^\ast$.
We would like to emphasize that the latter condition {\em is not}
the usual axiom for associative and unital 
$\ast$-algebras over $\bbC$, which is given by order-reversal 
$(a\,b)^\ast = b^\ast\,a^\ast$. As a consequence, our concept
of $\ast$-monoids given in Definition \ref{def:astmonoids}
{\em does not} include the usual associative and unital 
$\ast$-algebras over $\bbC$ as examples. 
We will show later in Example \ref{ex:reversal}
that the usual associative and unital $\ast$-algebras over $\bbC$
are recovered as $\ast$-algebras over a suitable $\ast$-operad,
which provides a sufficiently flexible framework to implement 
order-reversal $(a\,b)^\ast = b^\ast\,a^\ast$.
\end{ex}


\section{\label{sec:symseq}Involutive structures on colored symmetric sequences}
Colored operads can be defined as monoids in the monoidal category
of colored symmetric sequences, see e.g.\ \cite{Yau,WhiteYau,YauQFT,GambinoJoyal} 
and below for a brief review. Let $\CCC\in\Set$ be any 
non-empty set and $\MM$ any cocomplete closed symmetric monoidal category.
(We denote the monoidal structure on $\MM$ by $\otimes$ and $I$, 
and the internal hom by $[-,-] : \MM^\op\times\MM\to \MM$.) 
The category of {\em $\CCC$-colored symmetric sequences} with values in $\MM$
is defined as the functor category
\begin{flalign}\label{eqn:symseq}
\SymSeq_\CCC(\MM) \,:=\, \MM^{\Sigma_\CCC\times \CCC}\quad,
\end{flalign}
where $\Sigma_\CCC$ is the groupoid of $\CCC$-profiles defined in Example \ref{ex:Cprofiles}
and the set $\CCC$ is regarded as a discrete category. Given $X\in\SymSeq_\CCC(\MM)$,
we write
\begin{subequations}
\begin{flalign}
X\big(\substack{t\\\und{c}}\big)\in\MM
\end{flalign}
for the evaluation of this functor on objects $(\und{c},t)\in\Sigma_\CCC\times\CCC$ and
\begin{flalign}
X(\sigma)\,:\, X\big(\substack{t\\\und{c}}\big)~\longrightarrow~X\big(\substack{t\\\und{c}\sigma}\big)
\end{flalign}
\end{subequations}
for its evaluation on morphisms $\sigma : (\und{c},t) \to (\und{c}\sigma,t)$ in $\Sigma_\CCC\times\CCC$.
\sk

The category $\SymSeq_\CCC(\MM)$ can be equipped with the following monoidal structure:
The tensor product is given by the {\em circle product} $\circ :
\SymSeq_\CCC(\MM) \times \SymSeq_\CCC(\MM)\to \SymSeq_\CCC(\MM)$.
Concretely, the circle product of $X,Y\in \SymSeq_\CCC(\MM)$ is defined by the coend
\begin{flalign}\label{eqn:circleproduct}
(X\circ Y)\big(\substack{t\\\und{c}}\big)\,:=\, \int^{\und{a}}\int^{(\und{b}_1,\dots,\und{b}_m)}
 \Sigma_\CCC\big(\und{b}_1\otimes\cdots\otimes \und{b}_m,\und{c}\big)\otimes X\big(\substack{t\\\und{a}}\big)
\otimes Y\big(\substack{a_1\\ \und{b}_1}\big)\otimes\cdots\otimes Y\big(\substack{a_m\\ \und{b}_m}\big)\quad,
\end{flalign}
for all $(\und{c},t)\in\Sigma_\CCC\times\CCC$. Two remarks are in order:
(1)~This expression makes use of the symmetric monoidal structure on $\Sigma_\CCC$
that we described in Example \ref{ex:Cprofiles3}. (2)~The tensor product between the $\Hom$-set
$ \Sigma_\CCC\big(\und{b}_1\otimes\cdots\otimes \und{b}_m,\und{c}\big)\in\Set$ 
and the object $X\big(\substack{t\\\und{a}}\big)\in\MM$ is given by the canonical $\Set$-tensoring
of $\MM$, i.e.\ $S\otimes m := \coprod_{s\in S} m$ for any $S\in\Set$ and $m\in \MM$.
The {\em circle unit} is the object $I_\circ \in \SymSeq_\CCC(\MM)$ defined by
\begin{flalign}\label{eqn:circleunit}
I_\circ\big(\substack{t\\\und{c}}\big) \,:=\,\Sigma_\CCC(t,\und{c})\otimes I\quad,
\end{flalign}
for all $(\und{c},t)\in\Sigma_\CCC\times\CCC$. 
\begin{propo}\label{propo:symseqmonoidal}
$(\SymSeq_\CCC(\MM),\circ,I_\circ)$ is a right closed monoidal category. 
\end{propo}

The aim of this section is to transfer these structures and 
results to the setting of involutive categories.

\subsection{Product-exponential $2$-adjunction}
Because the category of symmetric sequences \eqref{eqn:symseq}
is defined as a functor category, we shall start with developing 
a notion of functor categories in the involutive setting. 
For this we will first recall the relevant structures for 
ordinary category theory from a perspective that easily generalizes
to involutive category theory.
\sk

Let us denote by $\Cat \,\widetilde{\times}\, \Cat$ the $2$-category with objects
given by pairs $(\CC,\DD)$ of categories, morphisms given by pairs
$(F,G)$ of functors and $2$-morphisms given by pairs $(\zeta,\xi)$ 
of natural transformations, and all compositions given component-wise.
(We use the symbol $\widetilde{\times}$ to denote the above product $2$-category
because we reserve the symbol $\times$ for the $2$-functors defined below.)
Notice that taking products of categories, functors and
natural transformations defines a $2$-functor
\begin{flalign}\label{eqn:times2fun}
\times \,:\, \Cat \,\widetilde{\times}\, \Cat~\longrightarrow~\Cat\quad.
\end{flalign}
Let us denote by $\Cat^\op$ the opposite $2$-category, i.e.\ morphisms
$\CC\to \DD$ are functors $F :\DD\to \CC$ going in the opposite direction
and $2$-morphisms are not reversed.
We define the exponential $2$-functor
\begin{flalign}\label{eqn:exp2fun}
(-)^{(-)}\,:\, \Cat^\op\,\widetilde{\times}\, \Cat~\longrightarrow~\Cat
\end{flalign}
as follows: 
\begin{itemize}
\item a pair $(\DD,\CC)$ of categories is mapped to the functor category $\CC^\DD$;
\item a pair $(G: \DD^\prime\to \DD,F:\CC\to\CC^\prime)$ of functors is 
mapped to the functor $F^G : \CC^\DD \to {\CC^\prime}^{\DD^\prime}$ that acts 
on objects and morphisms as
\begin{subequations}
\begin{flalign}
F^G\big(X : \DD\to\CC\big) \,&:=\, (FXG : \DD^\prime\to \CC^\prime)\quad,\\
F^G\big(\alpha: X\to Y\big) \,&:=\, ( F\alpha G : FXG \to FYG)\quad;
\end{flalign}
\end{subequations}
\item a pair $(\xi : G\to G^\prime , \zeta : F\to F^\prime)$ of natural transformations
is mapped to the natural transformation $\zeta^\xi : F^G\to {F^{\prime}}^{G^\prime}$ 
with components given by any of the two compositions in the commutative square
\begin{flalign}
\xymatrix{
\ar@{-->}[drr]^-{~(\zeta^\xi)_{X}^{}}\ar[d]_-{FX\xi} FXG\ar[rr]^-{\zeta XG} && F^\prime XG\ar[d]^-{F^\prime X\xi}\\
FXG^\prime \ar[rr]_-{\zeta XG^\prime} && F^\prime X G^\prime
}
\end{flalign}
for all $X\in \CC^\DD$. 
\end{itemize}
The two $2$-functors $\times$ and $(-)^{(-)}$ are related by a 
family of $2$-adjunctions.
\begin{propo}\label{propo:prodexpadjunction}
For every $\DD\in\Cat$, there is a $2$-adjunction
\begin{flalign}
\xymatrix{
(-)\times \DD \,:\, \Cat~\ar@<0.5ex>[r]&\ar@<0.5ex>[l]  ~\Cat \,:\, (-)^\DD
}\quad.
\end{flalign}
\end{propo}
\begin{proof}
The component at $\CC\in\Cat$ of the unit $2$-natural transformation
$\eta : \ID_{\Cat} \to ((-)\times \DD)^\DD$  is given by the functor
\begin{flalign}
\eta_\CC \,:\,\CC~\longrightarrow~(\CC\times\DD)^\DD
\end{flalign}
that assigns to $c\in\CC$ the inclusion functor $\eta_\CC(c) : \DD\to \CC\times \DD$ 
specified by $d \mapsto (c,d)$. The component at $\CC\in\Cat$ of the 
counit $2$-natural transformations
$\epsilon : (-)^\DD\times \DD \to \ID_{\Cat}$ is given by the evaluation functor
\begin{flalign}
\epsilon_\CC \,:\, \CC^\DD\times \DD~\longrightarrow~\CC\quad,
\end{flalign}
that assigns to $(X,d) \in \CC^\DD \times \DD$ the object $Xd \in \CC$. The triangle identities are a straightforward check.
\end{proof}

Because of their $2$-functoriality, our constructions above can be immediately extended 
to involutive category theory. Concretely, using the $2$-functor \eqref{eqn:times2fun}, 
we define the product $2$-functor
\begin{flalign}\label{eqn:timesICat2fun}
\times \,:\, \ICat \,\widetilde{\times}\, \ICat~\longrightarrow~\ICat
\end{flalign}
in the involutive setting as follows: 
\begin{itemize}
\item a pair of involutive categories is mapped to the involutive category
\begin{flalign}
(\CC,J,j)\times (\DD,K,k) &\,:=\, (\CC\times \DD,J\times K, j\times k)\quad;
\end{flalign}
\item a pair of involutive functors is mapped to the involutive functor
\begin{flalign}
(F,\nu)\times(G,\chi) \,:= \, (F\times G,\nu\times \chi)\quad;
\end{flalign}
\item a pair of involutive natural transformations is mapped to the involutive natural transformation
$\zeta\times\xi$.
\end{itemize}
Similarly, using the $2$-functor \eqref{eqn:exp2fun}, 
we define the  exponential $2$-functor
\begin{flalign}\label{eqn:expICat2fun}
(-)^{(-)}\,:\, \ICat^\op\,\widetilde{\times}\, \ICat~\longrightarrow~\ICat
\end{flalign}
in the involutive setting as follows: 
\begin{itemize}
\item a pair of involutive categories is mapped to the involutive category
\begin{flalign}
(\CC,J,j)^{(\DD,K,k)}\,:=\, \big(\CC^\DD, J^K,j^k\big)\quad;
\end{flalign}
\item a pair of involutive functors is mapped to the involutive functor
\begin{flalign}
(F,\nu)^{(G,\chi)}\, :=\, \big( F^G,\nu^{\chi^{-1}}  \big) \quad;
\end{flalign}
\item a pair of involutive natural transformations is mapped to the involutive natural transformation 
$\zeta^\xi$.
\end{itemize}
Analogously to Proposition \ref{propo:prodexpadjunction}, one can prove
\begin{propo}
For every $(\DD,K,k)\in\ICat$, there is a $2$-adjunction
\begin{flalign}
\xymatrix{
(-)\times( \DD,K,k) \,:\, \ICat~\ar@<0.5ex>[r]&\ar@<0.5ex>[l]  ~\ICat \,:\, (-)^{(\DD,K,k)}
}\quad.
\end{flalign}
\end{propo}

\subsection{Involutive colored symmetric sequences}
Let $(\MM,J,j)$ be an involutive closed symmetric monoidal category,
which we assume to be cocomplete, and $\CCC\in\Set$ a non-empty set of colors.
In order to define an involutive analog of the category of 
symmetric sequences \eqref{eqn:symseq}, one has to endow $\Sigma_\CCC \times \CCC$
with the structure of an involutive category.
The simplest possible choice is the trivial involutive structure from Example \ref{ex:trivial}, i.e.\
$\mathrm{triv}(\Sigma_\CCC \times \CCC) = 
(\Sigma_\CCC\times \CCC , \ID_{\Sigma_\CCC\times\CCC},\id_{ \ID_{\Sigma_\CCC\times\CCC}})\in\ICat$.
In particular, there is no non-trivial interplay between the involution functor
and the ordering of $\CCC$-profiles $\und{c} = (c_1,\dots,c_n)$. An alternative
choice that does describe a non-trivial interplay between involution and ordering 
of $\CCC$-profiles is obtained by considering the involutive symmetric monoidal category
$(\Sigma_\CCC,\mathrm{Rev}, \id_{\ID_{\Sigma_\CCC}})$ from 
Examples \ref{ex:Cprofiles}, \ref{ex:Cprofiles2} and \ref{ex:Cprofiles3},
where the involution functor is given by order-reversal. Endowing the discrete
category $\CCC$ with the trivial involutive structure and using the 
product $2$-functor \eqref{eqn:timesICat2fun},
we may form the involutive category $(\Sigma_\CCC,\mathrm{Rev}, 
\id_{\ID_{\Sigma_\CCC}})\times \mathrm{triv}(\CCC)\in\ICat$.
Both of these natural choices lead to the same theory of involutive 
colored sequences. Indeed, by Corollary \ref{cor:SigmaCCCtrivial}, 
there exists an $\ISMonCat$-isomorphism
$(\Sigma_{\CCC}, \mathrm{Rev}, \id_{\ID_{\Sigma_{\CCC}}}) \cong \triv(\Sigma_\CCC)$, 
which implies that $(\Sigma_\CCC,\mathrm{Rev}, \id_{\ID_{\Sigma_\CCC}})\times \mathrm{triv}(\CCC)
\cong \triv(\Sigma_\CCC)\times \triv(\CCC)\cong \mathrm{triv}(\Sigma_\CCC \times \CCC)$ in $\ICat$.
This motivates the following
\begin{defi}\label{def:invsymseq}
Let $\CCC\in\Set$ be a non-empty set.
The {\em involutive category of $\CCC$-colored symmetric sequences}
with values in a cocomplete involutive closed symmetric monoidal category
$(\MM,J,j)$ is defined via the exponential $2$-functor \eqref{eqn:expICat2fun} by
\begin{subequations}\label{eqn:symseqinvolutive}
\begin{flalign}
\big(\SymSeq_\CCC(\MM) ,J_\ast,j_\ast\big)~ :=~ (\MM,J,j)^{\triv(\Sigma_\CCC\times \CCC)}\quad.
\end{flalign}
Concretely, the endofunctor 
\begin{flalign}
J_\ast\, :=\, J^{\ID_{\Sigma_\CCC\times \CCC}} \,:\, \SymSeq_\CCC(\MM) ~\longrightarrow~\SymSeq_\CCC(\MM)
\end{flalign}
is given by post-composition with $J: \MM \to \MM$, i.e.\ $X\mapsto JX$,
and the natural isomorphism 
\begin{flalign}
j_\ast \,:=\, j^{\id_{\ID_{\Sigma_\CCC\times\CCC}}} \,:\, \ID_{\SymSeq_\CCC(\MM)} ~ \longrightarrow ~
J_\ast^2
\end{flalign} 
\end{subequations}
has components  ${j_\ast}_X := jX$ given by whiskering the natural isomorphism $j:\ID_\MM\to J^2$ 
and the functor $X : \Sigma_\CCC\times\CCC\to\MM$, for all $X \in \SymSeq_\CCC(\MM)$.
\end{defi}

We now show that the involutive category $(\SymSeq_\CCC(\MM) ,J_\ast,j_\ast)$ 
given in \eqref{eqn:symseqinvolutive} may be promoted to an involutive
monoidal category, extending the monoidal structure of Proposition 
\ref{propo:symseqmonoidal} to the involutive setting.
Recalling Definition \ref{def:invmoncat}, this amounts to
endowing the endofunctor $J_\ast : \SymSeq_\CCC(\MM) \to \SymSeq_\CCC(\MM)$ 
with the structure of a monoidal functor such that $j_\ast : \ID_{\SymSeq_\CCC(\MM)} \to J_\ast^2$ 
becomes a monoidal natural isomorphism.
We first define the natural transformation 
${J_\ast}_2 : \circ \, (J_\ast \times J_\ast) \to J_\ast\,\circ$ in terms of the components
\begin{flalign}\label{eqn:symseqJ2}
\xymatrix@C=2em{
\ar[d]_-{({J_\ast}_2)_{X,Y}}\big(J_\ast X\circ J_\ast Y\big)\big(\substack{t\\ \und{c}}\big) \ar@{=}[r] & \int\limits^{\und{a}}\int\limits^{(\und{b}_1,\dots,\und{b}_m)}  \Sigma_\CCC\big(\und{b}_1\otimes\cdots\otimes \und{b}_m,\und{c}\big)\otimes JX\big(\substack{t\\\und{a}}\big)
\otimes \Motimes_{i=1}^m JY\big(\substack{a_i\\ \und{b}_i}\big)
\ar[d]^-{\int \int \id\otimes J_2^m}  \\
\big(J_\ast(X\circ Y)\big)\big(\substack{t\\ \und{c}}\big) \ar[r]_-{\cong} &  \int\limits^{\und{a}}\int\limits^{(\und{b}_1,\dots,\und{b}_m)} 
\Sigma_\CCC\big(\und{b}_1\otimes\cdots\otimes \und{b}_m,\und{c}\big)\otimes J\Big( X\big(\substack{t\\\und{a}}\big)
\otimes \Motimes_{i=1}^m Y\big(\substack{a_i\\ \und{b}_i}\big)\Big)
}
\end{flalign}
for all $X,Y\in\SymSeq_\CCC(\MM)$ and all $(\und{c},t)\in\Sigma_\CCC\times\CCC$.
For the horizontal arrows we used the definition of the circle product \eqref{eqn:circleproduct} 
and the fact that $J : \MM\to \MM$ is self-adjoint (cf.\ Lemma \ref{lem:Jselfadjoint}), hence
it preserves coends and the $\Set$-tensoring. In the right vertical arrow
we denoted  by $J_2^m$ the $m$-fold iteration of the natural transformation 
$J_2 : \otimes \,(J\times J) \to J\,\otimes$ corresponding to the 
involutive symmetric monoidal category $(\MM,J,j)$. We next define
the $\SymSeq_\CCC(\MM)$-morphism ${J_\ast}_0 : I_\circ \to J_\ast  I_\circ$ for the circle unit \eqref{eqn:circleunit} by
\begin{flalign}\label{eqn:symseqJ0}
\xymatrix{
\ar@{=}[d]I_\circ\big(\substack{t\\\und{c}}\big) \ar[rr]^-{{J_\ast}_0}&& (J_\ast I_\circ )\big(\substack{t\\ \und{c}}\big)\ar[d]^-{\cong}\\
\Sigma_\CCC(t,\und{c})\otimes I \ar[rr]_-{\id\otimes J_0}&& \Sigma_\CCC(t,\und{c})\otimes JI
}
\end{flalign}
for all $(\und{c},t)\in \Sigma_\CCC\times\CCC$.
For the right vertical arrow we used again that $J:\MM\to\MM$ 
is self-adjoint and hence it preserves the $\Set$-tensoring.
In the bottom horizontal arrow $J_0 : I \to JI$ denotes the morphism
corresponding to the involutive symmetric monoidal category $(\MM,J,j)$.

\begin{theo}\label{theo:symseqinvmoncat}
The involutive category $(\SymSeq_\CCC(\MM),J_\ast,j_\ast)$ 
of \eqref{eqn:symseqinvolutive}  
becomes an involutive right closed monoidal category 
when the underlying category $\SymSeq_\CCC(\MM)$ is equipped 
with the circle monoidal structure of Proposition \ref{propo:symseqmonoidal} 
and the underlying endofunctor $J_\ast$ is equipped with the monoidal functor 
structure $({J_\ast}_2,{J_\ast}_0)$ of \eqref{eqn:symseqJ2} and \eqref{eqn:symseqJ0}. 
\end{theo}
\begin{proof}
It is straightforward to confirm that
$(J_\ast,{J_\ast}_2,{J_\ast}_0): \SymSeq_\CCC(\MM)\to \SymSeq_\CCC(\MM)$, as defined in
\eqref{eqn:symseqinvolutive}, \eqref{eqn:symseqJ2} and \eqref{eqn:symseqJ0}, 
is a monoidal endofunctor with respect to the circle monoidal structure 
and that the natural isomorphism $j_\ast : \ID_{\SymSeq_\CCC(\MM)}\to J_\ast^2$ is monoidal. 
\end{proof}

\begin{rem}\label{rem:whynonrev}
Because $\SymSeq_\CCC(\MM)$ in general does not admit a braiding, 
the {\em non-reversing} notion of involutive monoidal category 
due to \cite{Jacobs} (see also Definition \ref{def:invmoncat}) and 
the {\em reversing} one considered in \cite{Egger, BeggsMajid} are a priori inequivalent.
This is indeed the case: 
While Theorem \ref{theo:symseqinvmoncat} equips the monoidal category 
$\SymSeq_\CCC(\MM)$ with a {\em non-reversing} involutive structure, 
one cannot obtain a {\em reversing} one 
as this requires to specify isomorphisms 
$J_\ast X \circ J_\ast Y \cong J_\ast(Y \circ X)$,
which in general do not exist by the following argument:
Assume that $I \not \cong \emptyset$ in $\MM$ 
(e.g.\ $\MM = \Vec_\bbC$) and that 
the set $\CCC$ has cardinality $\geq 2$. Define 
$X,Y\in \SymSeq_\CCC(\MM)$ by setting
\begin{flalign}
X\big(\substack{t\\\und{c}}\big) = \big(\Sigma_{\CCC}(t,t_0)\times \Sigma_{\CCC}(t_0,\und{c})\big)\otimes \1\quad,\qquad
Y\big(\substack{t\\\und{c}}\big) = \Sigma_{\CCC}(\emptyset,\und{c})\otimes I\quad,
\end{flalign}
for some fixed $t_0\in\CCC$. Recalling \eqref{eqn:circleproduct} we obtain 
\begin{flalign}
(X\circ Y)\big(\substack{t\\\und{c}}\big) 
~\cong~\Sigma_{\CCC}(t,t_0) \otimes Y\big(\substack{t\\ \und{c}}\big)
 \quad,\qquad 
 (Y\circ X)\big(\substack{t\\\und{c}}\big)~\cong~ Y\big(\substack{t\\ \und{c}}\big) \quad.
\end{flalign}
Since $J_\ast X \cong X$ and $J_\ast Y \cong Y$, 
we find for $t\neq t_0$ that 
$(J_\ast X\circ J_\ast Y)\big(\substack{t\\\emptyset}\big) 
\cong \emptyset \not\cong I \cong 
J_\ast (Y\circ X)\big(\substack{t\\\emptyset}\big)$.
This counterexample explains 
why the {\em non-reversing} involutive structures defined 
by \cite{Jacobs} are better suited for developing
the theory of colored $\ast$-operads 
than the {\em reversing} ones of \cite{Egger, BeggsMajid}. 
\end{rem}

Many interesting constructions in colored operad theory
arise from changing the underlying set of colors, see e.g.\
\cite{BeniniSchenkelWoike} for examples inspired by quantum field theory.
We shall now generalize the relevant constructions to the setting of involutive 
category theory.
\sk

Any map $f :\CCC\to\DDD$ of non-empty sets induces a functor
$f :  \Sigma_{\CCC} \to \Sigma_{\DDD}$ between the associated groupoids of profiles. 
Concretely, we have that $\und{c} = (c_1,\dots,c_n)\mapsto f(\und{c}) = (f(c_1),\dots,f(c_n))$.
This functor may be equipped with the obvious involutive 
symmetric monoidal structure such that it defines
an involutive symmetric monoidal functor
\begin{flalign}
(f,\id_f) \,:\, (\Sigma_\CCC,\ID_{\Sigma_\CCC},\id_{\ID_{\Sigma_\CCC}})~\longrightarrow
~(\Sigma_\DDD,\ID_{\Sigma_\DDD},\id_{\ID_{\Sigma_\DDD}})\quad.
\end{flalign}
Moreover, regarding $\CCC$ and $\DDD$ as discrete categories,
we obtain an involutive functor (denoted by the same symbol)
\begin{flalign}
(f,\id_f)\,:\, (\CCC,\ID_\CCC,\id_{\ID_\CCC})~\longrightarrow ~
(\DDD,\ID_{\DDD},\id_{\ID_{\DDD}})
\end{flalign}
between the associated trivial involutive categories. Using the product
and exponential $2$-functors (cf.\ \eqref{eqn:timesICat2fun}  and \eqref{eqn:expICat2fun}),
we may exponentiate the identity $\ID_{(\MM,J,j)} = (\ID_\MM,\id_J) $ involutive functor
by the product involutive functor $(f,\id_f) \times (f,\id_f) $ to obtain an involutive functor
\begin{flalign}\label{eqn:fastINVsymseq}
(f^\ast, \id_{f^\ast J_\ast}) \,:\, \big(\SymSeq_\DDD(\MM),J_\ast,j_\ast\big) ~
\longrightarrow ~\big(\SymSeq_\CCC(\MM),J_\ast,j_\ast \big) \quad\quad
\end{flalign}
describing the pullback along $f$ of $\DDD$-colored symmetric sequences
to $\CCC$-colored symmetric sequences. (Notice that $ f^\ast J_\ast = J_\ast f^\ast $
as functors from $\SymSeq_\DDD(\MM)$ to $\SymSeq_\CCC(\MM)$ because
$J_\ast$ is a pushforward and $f^\ast$ is a pullback.)
\begin{cor}\label{cor:coloradjunction}
For every map $f : \CCC\to\DDD$ between non-empty sets,
there exists an involutive adjunction (cf.\ Definition \ref{def:invadj}) 
\begin{flalign}
\xymatrix{
(f_!,\lambda_f) \,:\, \big(\SymSeq_\CCC(\MM),J_\ast,j_\ast \big) ~\ar@<0.5ex>[r]&\ar@<0.5ex>[l]  ~\big(\SymSeq_\DDD(\MM),J_\ast,j_\ast\big) \,:\, (f^\ast, \id_{f^\ast J_\ast })
}\quad.
\end{flalign}
\end{cor}
\begin{proof}
By left Kan extension, the functor $f^\ast$ has a left adjoint $f_!$. The involutive structure
$\lambda_f$ on $f_!$ is the one described in Proposition \ref{propo:invadj}, which implies that
we have an involutive adjunction.
\end{proof}

The pullback functor $f^\ast: \SymSeq_\DDD(\MM) \to \SymSeq_\CCC(\MM)$ 
may be equipped with the following canonical monoidal structure:
The components of the natural transformation $f^\ast_2: \circ^\CCC 
(f^\ast \times f^\ast) \to f^\ast \circ^\DDD$ are specified by
\begin{flalign}\label{eqn:fast2}
\xymatrix{
\Sigma_\CCC \big( \und{b}_1 \otimes \cdots \otimes \und{b}_m, \und{c} \big) 
\otimes f^\ast X \big( \substack{ t \\ \und{a} } \big)
\otimes \bigotimes_{i=1}^m f^\ast Y \big( \substack{ a_i \\ \und{b}_i } \big) 
\ar[d]_-{f \otimes \id} \ar[r]
& \big( f^\ast X \circ^\CCC f^\ast Y \big) 
\big( \substack{ t \\ \und{c} } \big) \ar[d]^-{(f^\ast_2)_{X,Y}} \\ 
\Sigma_\DDD \big( f(\und{b}_1) \otimes \cdots \otimes f(\und{b}_m), 
f(\und{c}) \big) \otimes f^\ast X \big( \substack{ t \\ \und{a} } \big)
\otimes \bigotimes_{i=1}^m f^\ast Y \big( \substack{ a_i \\ \und{b}_i } \big)  
\ar[r] 
& \big( f^\ast (X \circ^\DDD Y) \big) \big( \substack{ t \\ \und{c} } \big) 
}
\end{flalign}
for all $X,Y \in \SymSeq_\DDD(\MM)$ and all $(\und{c},t) \in \Sigma_{\CCC} \times \CCC$.
The horizontal arrows are the canonical inclusions into the coend and the
left vertical arrow denotes the action of the functor $f : \Sigma_\CCC\to\Sigma_\DDD$ on
$\Hom$-sets. The $\SymSeq_\CCC(\MM)$-morphism $f^\ast_0: I_\circ^\CCC \to f^\ast I_\circ^\DDD$ 
is defined similarly by
\begin{flalign}\label{eqn:fast0}
\xymatrix{
f^\ast_0 \,:\, \1_\circ^\CCC\big(\substack{t\\\und{c}}\big) = \Sigma_\CCC(t,\und{c})\otimes \1 ~\ar[rr]^-{f\otimes\id}&&~
\Sigma_\DDD\big(f(t),f(\und{c})\big)\otimes \1 = f^\ast(\1_\circ^\DDD)\big(\substack{t\\\und{c}}\big)
}\quad,
\end{flalign}
for all $(\und{c},t) \in \Sigma_{\CCC} \times \CCC$. 

\begin{theo}\label{theo:colorpullbackSymSeq}
For every map $f : \CCC\to\DDD$ between non-empty sets,
the involutive functor $(f^\ast, \id_{f^\ast J_\ast }) : (\SymSeq_\DDD(\MM),J_\ast,j_\ast) 
\to (\SymSeq_\CCC(\MM),J_\ast,j_\ast)$ of \eqref{eqn:fastINVsymseq}
becomes an involutive monoidal functor when equipped with the 
monoidal structure $(f^\ast_2,f^\ast_0)$ of \eqref{eqn:fast2} and \eqref{eqn:fast0}. 
\end{theo}
\begin{proof}
By Definition \ref{def:invmonfunandnat}, it remains to prove that
$\id_{J_\ast f^\ast} :  J_\ast f^\ast \to f^\ast J_\ast = J_\ast f^\ast$ is a monoidal natural transformation,
which is clearly the case.
\end{proof}

\subsection{$\ast$-objects}
We conclude this section by describing rather explicitly the
monoidal category 
\begin{flalign}\label{eqn:astObjSymSeq}
\astObj\big(\SymSeq_\CCC(\MM),J_\ast,j_\ast\big)\in\MonCat
\end{flalign}
of $\ast$-objects in the involutive monoidal category of symmetric sequences.
Given any $\ast$-object $(\ast : X\to J_\ast X)\in \astObj\big(\SymSeq_\CCC(\MM),J_\ast,j_\ast\big)$,
we consider its components at $(\und{c},t)\in\Sigma_\CCC\times\CCC$
and observe that this is precisely the same data as a symmetric sequence
with values in $\astObj(\MM,J,j)$, which is a cocomplete closed symmetric 
monoidal category, cf.\ Proposition \ref{propo:astObjmonoidal} and Remark \ref{rem:astObjlim}.
Similarly, one observes that a morphism in \eqref{eqn:astObjSymSeq}
is the same data as a morphism in $\SymSeq_\CCC(\astObj(\MM,J,j))$,
which means that these two categories are canonically isomorphic.
We now show that this isomorphism is compatible with the monoidal structures.
\begin{propo}\label{propo:astObjSymSeq}
The canonical identification above defines an isomorphism 
\begin{flalign}\label{eqn:astObjSymSeq2perspectives}
\astObj \big( \SymSeq_\CCC(\MM), J_\ast, j_\ast \big) 
~\cong~\SymSeq_\CCC \big( \astObj(\MM,J,j) \big) 
\end{flalign}
of monoidal categories. 
\end{propo}
\begin{proof}
It remains to prove that our canonical isomorphism of categories is monoidal,
i.e.\ that tensor products and units are preserved up to coherent isomorphisms.
Given two objects $\ast : X\to J_\ast X$ and $\ast^\prime : Y\to J_\ast Y$
in $\astObj \big( \SymSeq_\CCC(\MM), J_\ast, j_\ast \big)$, their tensor product
reads as
\begin{flalign}
\big(\ast : X\to J_\ast X\big)\circ \big(\ast^\prime: Y\to J_\ast Y\big)
\,=\, \Big(\xymatrix@C=3em{
X\circ Y \ar[r]^-{\ast\circ\ast^\prime} & J_\ast X\circ J_\ast Y \ar[r]^-{({J_\ast}_2)_{X,Y}} & J_\ast (X\circ Y)
}\Big)\quad.
\end{flalign}
By a brief calculation one shows that the composed morphism on the right-hand side
of this equation is induced by functoriality of coends and $\Set$-tensoring via the family of maps
\begin{flalign}\label{eqn:XYYYtimestmp}
\xymatrix@C=2em{X\big(\substack{t\\\und{a}}\big)\otimes \Motimes_{i=1}^m Y\big(\substack{a_i\\\und{b}_i}\big)
\ar[rr]^-{\ast\otimes\Motimes_{i=1}^m \ast^\prime} &&
JX\big(\substack{t\\\und{a}}\big)\otimes \Motimes_{i=1}^m JY\big(\substack{a_i\\\und{b}_i}\big)\ar[r]^-{J_2^m}
&J\big(X\big(\substack{t\\\und{a}}\big)\otimes \Motimes_{i=1}^m Y\big(\substack{a_i\\\und{b}_i}\big)\big)
}\quad.
\end{flalign}
Notice that \eqref{eqn:XYYYtimestmp} is the tensor product
$\big(\ast : X\big(\substack{t\\\und{a}}\big)\to JX\big(\substack{t\\\und{a}}\big)\big) \otimes\Motimes_{i=1}^m
\big(\ast^\prime : Y\big(\substack{a_i\\\und{b}_i}\big)\to JY\big(\substack{a_i\\\und{b}_i}\big)\big)$
in $\astObj(\MM,J,j)$. Because $J$ preserves coends and the $\Set$-tensoring,
we obtain the natural isomorphism relating the tensor products on both sides of \eqref{eqn:astObjSymSeq2perspectives}.
A similar construction provides the isomorphism relating the units.
\end{proof}


\section{\label{sec:astOp}Colored $\ast$-operads}
Let $\CCC\in\Set$ be any non-empty set. We briefly recall the concept of $\CCC$-colored operads. 
\begin{defi}\label{def:Op}
The category of {\em $\CCC$-colored operads} with values in a cocomplete 
closed symmetric monoidal category $\MM$ is the category of monoids 
(cf.\ Definition \ref{def:monoid}) in the monoidal category $\SymSeq_\CCC(\MM)$ 
(cf.\ \eqref{eqn:symseq}, \eqref{eqn:circleproduct} and \eqref{eqn:circleunit}), i.e.\
\begin{flalign}
\Op_\CCC(\MM)~:=~ \Mon\big(\SymSeq_\CCC(\MM)\big) \quad.
\end{flalign}
\end{defi}

Using the concepts and techniques that we have developed so far in this paper,
the above definition admits the following natural generalization to involutive category theory.
\begin{defi}\label{def:astOp}
The category of {\em $\CCC$-colored $\ast$-operads} with values in a cocomplete involutive 
closed symmetric monoidal category $(\MM,J,j)$ is the category of $\ast$-monoids (cf.\ Definition \ref{def:astmonoids})
in the involutive monoidal category $(\SymSeq_\CCC(\MM),J_\ast,j_\ast)$ 
(cf.\ Theorem \ref{theo:symseqinvmoncat}), i.e.\
\begin{flalign}
\astOp_\CCC(\MM,J,j)~:=~ \astMon\big(\SymSeq_\CCC(\MM),J_\ast,j_\ast \big) \quad.
\end{flalign}
\end{defi}

\begin{rem}\label{rem:astOp}
It is worth to specialize Remark \ref{rem:astmonoids} to the present case. 
We observe that a $\CCC$-colored $\ast$-operad is a quadruple $(\O,\gamma,\oone,\ast)$ 
consisting of a $\CCC$-colored symmetric sequence $\O\in\SymSeq_\CCC(\MM)$ 
and three $\SymSeq_\CCC(\MM)$-morphisms $\gamma : \O\circ \O\to \O$ 
(called {\em operadic composition}), $\oone : I_\circ \to \O$ (called {\em operadic unit}) 
and $\ast : \O\to J_\ast\O$ (called {\em $\ast$-involution}), 
which satisfy the following conditions: 
\begin{itemize}
\item[(1)] $(\O,\gamma,\oone)$ is a monoid in $(\SymSeq_\CCC(\MM),\circ,I_\circ)$, i.e.\ the diagrams 
\begin{flalign}
\xymatrix{
\ar[d]_-{\gamma\circ\id}(\O\circ\O)\circ\O \ar[r]^-{\cong} & \O\circ(\O\circ\O) \ar[r]^-{\id\circ\gamma} & \O\circ \O\ar[d]^-{\gamma}\\
\O\circ\O \ar[rr]_-{\gamma}&& \O
}\qquad\quad
\xymatrix{
\ar[dr]_-{\cong}I_\circ \circ\O \ar[r]^-{\oone\circ \id} & \O\circ\O \ar[d]^-{\gamma}& \ar[l]_-{\id\circ\oone} \O\circ I_\circ\ar[dl]^-{\cong}\\
&\O&
}
\end{flalign}
in $\SymSeq_\CCC(\MM)$ commute;

\item[(2)] $\ast: \O \to J_\ast\O$ is a $\ast$-object in $(\SymSeq_\CCC(\MM),J_\ast,j_\ast)$, 
i.e.\ the diagram 
\begin{flalign}
\xymatrix{
\ar[dr]_-{(j_\ast)_\O}\O \ar[r]^{\ast} & J_\ast\O\ar[d]^-{J_\ast\ast}\\
&J_\ast^2\O
}
\end{flalign}
in $\SymSeq_\CCC(\MM)$ commutes; 

\item[(3)] these two structures are compatible, i.e.\ the diagrams 
\begin{flalign}
\xymatrix{
\ar[d]_-{{J_\ast}_0}I_\circ \ar[r]^-{\oone} & \O\ar[d]^-{\ast}\\
J_\ast I_\circ \ar[r]_-{J_\ast \oone}& J_\ast\O
}\qquad\qquad
\xymatrix@C=3.5em{
\ar[d]_-{\gamma}\O\circ\O\ar[r]^-{\ast\circ\ast} & J_\ast\O\circ J_\ast\O\ar[r]^-{({J_\ast}_2)_{\O,\O}} & J_\ast(\O\circ\O)\ar[d]^-{J_\ast\gamma}\\
\O \ar[rr]_-{\ast}&&J_\ast\O
}
\end{flalign}
in $\SymSeq_\CCC(\MM)$ commute. 
\end{itemize}
In particular, there exist two equivalent interpretations of a colored $\ast$-operad: 
The first option is to regard $(\O,\gamma,\oone)$ 
as an ordinary $\CCC$-colored operad valued in $(\MM,J,j)$, equipped with 
an operad morphism $\ast : (\O,\gamma,\oone) \to \Mon(J_\ast)(\O,\gamma,\oone)$. 
The second option is to regard $\ast: \O \to J_\ast\O$ as a $\ast$-object 
in $(\SymSeq_\CCC(\MM),J_\ast,j_\ast)$, equipped with the structure of a monoid 
consisting of the $\ast$-morphisms 
$\gamma: (\ast: \O \to J_\ast\O) \circ (\ast: \O \to J_\ast\O) \to (\ast: \O \to J_\ast\O)$ 
and $\oone: ({J_\ast}_0: I_\circ \to J_\ast I_\circ) \to (\ast: \O \to J_\ast\O)$. 
\end{rem}

\begin{propo}\label{propo:OpastObj}
The category of $\CCC$-colored $\ast$-operads with values in a cocomplete involutive 
closed symmetric monoidal category $(\MM,J,j)$ is isomorphic to 
the category of $\CCC$-colored operads with values in the cocomplete closed symmetric
monoidal category $\astObj(\MM,J,j)$, 
i.e.\ there exists an isomorphism
\begin{flalign}
\astOp_\CCC(\MM,J,j)~ \cong~ \Op_\CCC \big( \astObj(\MM,J,j) \big) 
\end{flalign}
of categories. 
\end{propo}
\begin{proof}
This is proven by the following chain of $\Cat$-isomorphisms
\begin{flalign}\nn
\astOp_\CCC(\MM,J,j) 
& ~=~ \Mon \left( \astObj \big( \SymSeq_\CCC(\MM),J_\ast,j_\ast \big) \right) \\
& ~\cong~ \Mon \left( \SymSeq_\CCC \big( \astObj(\MM,J,j) \big) \right) ~=~ \Op_\CCC \big( \astObj(\MM,J,j) \big) \quad,
\end{flalign}
where in the first step we used Definitions \ref{def:astOp} and \ref{def:astmonoids},
in the second step Proposition \ref{propo:astObjSymSeq} and in the last step
Definition \ref{def:Op}.
\end{proof}

\begin{rem}\label{rem:OpastObj}
Proposition \ref{propo:OpastObj} may be summarized by the following slogan: 
Colored $\ast$-operads are the same things as colored operads in $\ast$-objects. 
We would like to stress that this result, whose proof
relies on the whole spectrum of techniques for involutive category theory
developed in \cite{Jacobs} and in the previous sections of the present paper, 
does not make the definition of operads as $\ast$-monoids unnecessary.
Being able to switch between these two equivalent perspectives on colored $\ast$-operads
is valuable for various reasons. On the one hand, when interpreted as ordinary
colored operads in $\astObj(\MM,J,j)$, it is straightforward to transfer structural results and techniques
from ordinary operad theory to involutive operad theory. On the other hand, when interpreted
according to Definition \ref{def:astOp} as $\ast$-monoids, it is relatively easy to equip known
examples of ordinary colored operads with a suitable $\ast$-involution, see Section \ref{sec:QFTs}
for a specific class of examples. Moreover, this perspective relates to the involutive monoid and monad 
theory initiated in \cite{Jacobs}, see also Section \ref{sec:astAlg} below.
\end{rem}

We shall now study the behavior of colored $\ast$-operads
under changing the underlying set of colors. Let $f :\CCC\to\DDD$
be a map between non-empty sets. By Theorem \ref{theo:colorpullbackSymSeq},
we obtain an involutive monoidal functor
$(f^\ast, \id_{f^\ast J_\ast }) : (\SymSeq_\DDD(\MM),J_\ast,j_\ast) 
\to (\SymSeq_\CCC(\MM),J_\ast,j_\ast)$. As a consequence of
$2$-functoriality of $\astMon : \IMonCat\to \Cat$ (cf.\ Definition \ref{def:astmonoids})
and the definition of colored $\ast$-operads (cf.\ Definition \ref{def:astOp}),
we obtain
\begin{propo}
For every map $f : \CCC\to\DDD$ between non-empty sets, there exists
a functor
\begin{flalign}
f^\ast := \astMon(f^\ast, \id_{f^\ast J_\ast })\,:\,  \astOp_{\DDD}(\MM, J,j) ~\longrightarrow~ \astOp_{\CCC}(\MM, J,j) \quad,
\end{flalign}
which we call the pullback functor.
\end{propo}

Using the pullback functor, we may define the category of $\ast$-operads with varying colors.
\begin{defi}\label{def:astOpcolor}
We denote by $\astOp(\MM,J,j)$ the category of {\em colored 
$\ast$-operads} with values in $(\MM,J,j)$. The objects are pairs
$( \CCC, \O)$ consisting of a non-empty set $\CCC\in\Set$ 
and a $\CCC$-colored $\ast$-operad $\O \in\astOp_\CCC(\MM,J,j)$. 
The morphisms are pairs $(f,\phi) : ( \CCC, \O ) \to ( \DDD, \P)$ 
consisting of a map $f :\CCC\to\DDD$ between non-empty sets
and a $\astOp_\CCC(\MM,J,j)$-morphism 
$\phi : \O \to f^\ast \P$.
\end{defi}

\begin{rem}
There exists a projection functor $\pi : \astOp(\MM,J,j)\to \Set$,
given explicitly by $( \CCC, \O) \mapsto \CCC$, whose fiber $\pi^{-1}(\CCC)$ over 
$\emptyset\neq \CCC\in\Set$ is isomorphic to the category 
$\astOp_\CCC(\MM,J,j)$ of $\CCC$-colored $\ast$-operads.
\end{rem}


\section{\label{sec:astAlg}$\ast$-algebras over colored $\ast$-operads}
A convenient description of algebras over colored operads
is in terms of algebras over their associated monads.
Let us briefly review the relevant constructions before generalizing them to 
the setting of involutive categories.
\sk

Let $\CCC\in\Set$ be a non-empty set of colors.
Recall that the category of {\em $\CCC$-colored objects} with values in $\MM$
is the functor category $\MM^\CCC$.
We may equivalently regard $\MM^\CCC$ as the full subcategory
of $\SymSeq_\CCC(\MM)$ consisting of all functors
$X : \Sigma_\CCC\times\CCC\to \MM$ such that $X\big(\substack{t\\\und{c}}\big) =\emptyset$, 
for all $(\und{c},t)\in\Sigma_\CCC\times\CCC$ with length $\vert\und{c}\vert \geq 1$.
We introduce the notation $X_t := X\big(\substack{t\\\emptyset}\big)$,
for all $t\in\CCC$.
\sk

Given any $\CCC$-colored operad $\O\in \Op_\CCC(\MM)$,
the endofunctor $\O\circ (-) : \SymSeq_\CCC(\MM)\to \SymSeq_\CCC(\MM)$
restricts to an endofunctor
\begin{flalign}\label{eqn:Ocircrestricted}
\O\circ(-) \,:\,\MM^\CCC~\longrightarrow~\MM^\CCC
\end{flalign}
on the category of colored objects. Because $\O$ is by definition a monoid in 
$\SymSeq_\CCC(\MM)$, with multiplication $\gamma$ and unit $\oone$,
it follows that \eqref{eqn:Ocircrestricted} canonically carries the structure of a
{\em monad} in the category $\MM^\CCC$. We refer to \cite[Chapter VI]{MacLane}
for details on monad theory. Concretely, the structure natural transformations
$\gamma : \O\circ (\O\circ (-)) \to \O\circ(-)$ and $\oone : \ID_{\MM^\CCC} \to \O\circ (-)$,
which we denote with abuse of notation by the same symbols as the operadic composition and unit,
are given by the components
\begin{flalign}\label{eqn:monadstructuremaps}
\xymatrix{
\O\circ(\O\circ X) \ar[r]^-{\gamma_X^{}} & \O\circ X\\
\ar[u]^-{\cong}(\O\circ\O)\circ X\ar[ru]_-{\gamma\circ\id}&
}\qquad
\xymatrix{
X\ar[r]^-{\oone_X^{}} & \O\circ X\\
\ar[u]^-{\cong}I_\circ\circ X\ar[ur]_-{\oone\circ \id}&
}
\end{flalign}
for all $X\in\MM^\CCC$.

\begin{defi}\label{def:alg}
The category $\Alg(\O)$ of {\em algebras over a $\CCC$-colored operad} 
$\O \in \Op_\CCC(\MM)$ is the category of algebras 
over the monad $\O \circ (-) : \MM^\CCC \to \MM^\CCC$.
Concretely, an object of $\Alg(\O)$ is a pair $(A,\alpha)$ 
consisting of an object $A \in \MM^\CCC$ and an $\MM^\CCC$-morphism 
$\alpha: \O \circ A \to A$ such that $\alpha\, (\O \circ \alpha) = \alpha\, \gamma_A$ 
and $\alpha\, \oone_A = \id_A$. 
An $\Alg(\O)$-morphism $\varphi: (A,\alpha) \to (B,\beta)$ 
is an $\MM^\CCC$-morphism $\varphi: A \to B$  that 
preserves the structure maps, i.e.\ 
$\beta\, (\O \circ \varphi) = \varphi\, \alpha$. 
\end{defi}

The assignment of the categories of algebras to colored operads
is functorial 
\begin{flalign}\label{eqn:Alg}
\Alg \,:\, \Op(\MM)^\op ~\longrightarrow ~\Cat
\end{flalign}
with respect to the category $\Op(\MM)$ 
of colored operads with varying colors. Concretely, given any $\Op(\MM)$-morphism
$(f,\phi) : (\CCC,\O)\to (\DDD,\P)$, i.e.\ a map of non-empty sets
$f:\CCC\to\DDD$ together with an $\Op_\CCC(\MM)$-morphism
$\phi : \O\to f^\ast \P$, we define a functor
\begin{subequations}\label{eqn:Algpullback}
\begin{flalign}
(f,\phi)^\ast:= \Alg(f,\phi)\,:\, \Alg(\P)~\longrightarrow~\Alg(\O)
\end{flalign}
by setting
\begin{flalign}
(f,\phi)^\ast\big(A,\alpha\big) ~:=~\Big(f^\ast A , \xymatrix@C=2.7em{\O\circ f^\ast A
\ar[r]^-{\phi\circ \id} & f^\ast\P\circ f^\ast A  \ar[r]^-{(f^\ast_2)_{\P,A}} & f^\ast(\P\circ A) \ar[r]^-{f^\ast \alpha} & f^\ast A
} \Big)\quad,
\end{flalign}
\end{subequations}
for all $\P$-algebras $(A , \alpha : \P\circ A\to A)\in \Alg(\P)$. (The natural transformation
$f_2^\ast$ was defined in \eqref{eqn:fast2}.) 
Furthermore, as a consequence of the adjoint lifting theorem \cite[Chapter 4.5]{handbook2},
it follows that the functor $(f,\phi)^\ast$ admits a left adjoint (called {\em operadic left Kan extension}), 
i.e.\ we obtain an adjunction
\begin{flalign}\label{eqn:changeOp}
\xymatrix{
(f,\phi)_! \,:\, \Alg(\O)~\ar@<0.5ex>[r]&\ar@<0.5ex>[l]  ~\Alg(\P) \,:\, (f,\phi)^\ast
}\quad,
\end{flalign}
for every $\Op(\MM)$-morphism $(f,\phi) : (\CCC,\O)\to (\DDD,\P)$.
See for example \cite{BergerMoerdijk,BeniniSchenkelWoike} for further details and 
also \cite{BeniniSchenkelWoike} for applications of these adjunctions
to quantum field theory. 
\sk

We develop now a generalization of these definitions and constructions
to the setting of involutive categories. Let $(\MM,J,j)$ be
a cocomplete involutive closed symmetric monoidal category.
The involutive analog of the category of $\CCC$-colored objects
is obtained by using the exponential $2$-functor \eqref{eqn:expICat2fun}
to form $(\MM,J,j)^{\mathrm{triv}(\CCC)}\in\ICat$.
Notice that the full subcategory embedding $\MM^\CCC\hookrightarrow \SymSeq_\CCC(\MM)$ 
can be equipped with an obvious involutive structure, 
thus providing an $\ICat$-isomorphism between
$(\MM,J,j)^{\mathrm{triv}(\CCC)}$ and the involutive
category obtained by restricting the involutive structure on
$(\SymSeq_\CCC(\MM),J_\ast,j_\ast)$ to the full subcategory
$\MM^\CCC\subseteq \SymSeq_\CCC(\MM)$. In the following we
shall always suppress this isomorphism and identify the involutive categories 
\begin{flalign}
(\MM^\CCC,J_\ast,j_\ast) ~\cong~(\MM,J,j)^{\mathrm{triv}(\CCC)}\quad.
\end{flalign}
Given a $\CCC$-colored $\ast$-operad
$\O\in\astOp_\CCC(\MM,J,j)$ in the sense of Definition \ref{def:astOp} 
(see also Remark \ref{rem:astOp} for a more explicit description), 
we obtain an involutive endofunctor
\begin{subequations}\label{eqn:involutivemonadfromastoperad}
\begin{flalign}
\big(\O\circ (-), \nu \big)\,:\, (\MM^\CCC,J_\ast,j_\ast)~\longrightarrow~(\MM^\CCC,J_\ast,j_\ast)\quad
\end{flalign}
with the natural transformation $\nu : \O\circ J_\ast(-) \to J_\ast(\O\circ(-))$ defined
by the components
\begin{flalign}
\xymatrix{
\ar[dr]_-{\ast \circ \id~~}\O\circ J_\ast X \ar[rr]^-{\nu_X} && J_\ast\big(\O\circ X\big)\\
&J_\ast\O\circ J_\ast X \ar[ru]_-{~~~({J_{\ast}}_2)_{\O,X}}&
}
\end{flalign}
\end{subequations}
for all $X\in\MM^\CCC$, where $\ast : \O\to J_\ast \O$ denotes the $\ast$-involution on $\O$.
\begin{propo}
Given any $\CCC$-colored $\ast$-operad $(\O,\gamma,\oone,\ast) \in\astOp_\CCC(\MM,J,j)$,
the components given in \eqref{eqn:monadstructuremaps} 
define involutive natural transformations $\gamma : (\O\circ (-), \nu )~ (\O\circ (-), \nu ) \to (\O\circ (-), \nu)$
and $\oone : (\ID_{\MM^\CCC},\id_{J_\ast}) \to (\O\circ (-), \nu )$ for the involutive endofunctor
\eqref{eqn:involutivemonadfromastoperad}. In the terminology of \cite[Definition 7]{Jacobs},
the triple $\big((\O\circ (-), \nu ),\gamma,\oone\big)$ is an {\em involutive monad}
in $(\MM^\CCC,J_\ast,j_\ast)$.
\end{propo}
\begin{proof}
This statement is analogous \cite[Example 3~(i)]{Jacobs} and may be proven by a slightly lengthy diagram chase argument.
\end{proof}

The category of algebras $\Alg(\O)$ (cf.\ Definition \ref{def:alg}) 
over (the underlying colored operad of) 
a $\CCC$-colored $\ast$-operad $\O\in\astOp_\CCC(\MM,J,j)$
can be equipped with a canonical involutive structure
\begin{flalign}\label{eqn:AlgOinvolutive}
\big(\Alg(\O),J_\O,j_\O\big)\in\ICat\quad,
\end{flalign} 
see also \cite[Proposition 3]{Jacobs} for a similar construction.
Concretely, the endofunctor $J_\O : \Alg(\O)\to\Alg(\O)$ acts on objects 
$(A,\alpha)\in\Alg(\O)$ as
\begin{flalign}
J_\O\big( A,\alpha \big)~:=~\Big(J_\ast A , 
\xymatrix@C=3.5em{
\O\circ J_\ast A \ar[r]^-{\ast\circ\id} &
J_\ast \O\circ J_\ast A\ar[r]^-{({J_\ast}_2)_{\O,A}} & 
J_\ast\big(\O\circ A\big)\ar[r]^-{J_\ast\alpha} & J_\ast A
}
\Big)
\end{flalign}
and on morphisms as $J_\ast$. The natural transformation
$j_\O : \ID_{\Alg(\O)} \to J_\O^2$ is defined by the components
${j_\O}_{(A,\alpha)} := {j_{\ast}}_A$, for all $(A,\alpha)\in\Alg(\O)$.
This allows us to introduce the concept of $\ast$-algebras
over colored $\ast$-operads.
\begin{defi}\label{def:astAlg}
The category of {\em $\ast$-algebras} over a $\CCC$-colored
$\ast$-operad $\O\in\astOp_\CCC(\MM,J,j)$ is defined
by evaluating the $2$-functor $\astObj : \ICat\to\Cat$ (cf.\ \eqref{eqn:astObj2fun})
on the involutive category of $\O$-algebras \eqref{eqn:AlgOinvolutive}, i.e.\
\begin{flalign}
\astAlg(\O)~:=~\astObj\big(\Alg(\O),J_\O,j_\O\big)\quad.
\end{flalign}
\end{defi}
\begin{rem}\label{rem:astAlgexplicit}
Unpacking this definition, we obtain that 
a $\ast$-algebra over $\O\in\astOp_\CCC(\MM,J,j)$ 
is a triple $(A,\alpha,\ast_A)\in\astAlg(\O)$ 
consisting of a $\CCC$-colored object $A \in \MM^\CCC$ 
and two $\MM^\CCC$-morphisms $\alpha: \O \circ A \to A$ 
and $\ast_A: A \to J_\ast A$, which satisfy the following conditions: 
\begin{itemize}
\item[(1)] $(A,\alpha) \in \Alg(\O)$ is an algebra over the $\CCC$-colored operad $\O$; 

\item[(2)] $(\ast_A: A \to J_\ast A) \in \astObj(\MM^\CCC,J_\ast,j_\ast)$ 
is a $\ast$-object in the involutive category $(\MM^\CCC,J_\ast,j_\ast)$; 

\item[(3)] these two structures are compatible, i.e.\ the diagram
\begin{flalign}
\xymatrix@C=3.5em{
\ar[d]_-{\O\circ \ast_A} \O\circ A \ar[rrr]^-{\alpha} &&& \ar[d]^-{\ast_A} A\\
\O\circ J_\ast A \ar[r]_-{\ast\circ \id}& J_\ast\O\circ J_\ast A \ar[r]_-{({J_\ast}_2)_{\O,A}}& J_\ast (\O\circ A) \ar[r]_-{J_\ast\alpha}& J_\ast A 
}
\end{flalign} 
in $\MM^\CCC$ commutes. 
\end{itemize}
A $\ast$-algebra morphism $\varphi : (A,\alpha,\ast_A)\to (B,\beta,\ast_B)$
is an $\MM^\CCC$-morphism $\varphi : A\to B$ preserving the structure maps
and $\ast$-involutions, i.e.\ $\beta\, (\O\circ\varphi) = \varphi\, \alpha$
and $\ast_B\, \varphi = (J_\ast\varphi)\,\ast_A$.
\end{rem}

Similarly to \eqref{eqn:Alg}, we observe that the 
assignment of the involutive categories of algebras 
to colored $\ast$-operads is functorial 
\begin{flalign}\label{eqn:AlgICatFunctor}
\Alg \,:\, \astOp(\MM,J,j)^\op~\longrightarrow~\ICat
\end{flalign}
with respect to the category  $\astOp(\MM,J,j)$ of colored $\ast$-operads with varying colors
(cf.\ Definition \ref{def:astOpcolor}). Concretely, this functor
assigns to a $\astOp(\MM,J,j)$-morphism $(f,\phi) : (\CC,\O)\to(\DD,\P)$
the involutive functor
\begin{flalign}\label{eqn:Algpullbackinv}
\big( (f,\phi)^\ast, \id_{(f,\phi)^\ast\, J_\P} \big) \,:\, 
\big( \Alg(\P), J_\P, j_\P \big) ~\longrightarrow~ 
\big( \Alg(\O), J_\O, j_\O \big) \quad, 
\end{flalign}
which is given by equipping the pullback functor \eqref{eqn:Algpullback}
with the trivial involutive structure $\id_{(f,\phi)^\ast\, J_\P}: 
(f,\phi)^\ast\, J_\P \to J_\O\, (f,\phi)^\ast = (f,\phi)^\ast\, J_\P$.
(Showing that $J_\O\, (f,\phi)^\ast = (f,\phi)^\ast\, J_\P$ requires a brief check.)
As a consequence of \eqref{eqn:AlgICatFunctor} and ($2$-)functoriality of 
$\astObj: \ICat \to \Cat$ (cf.\ \eqref{eqn:astObj2fun}), we obtain that also the
assignment of the categories of $\ast$-algebras (cf.\ Definition \ref{def:astAlg}) 
to colored $\ast$-operad is functorial
\begin{flalign}\label{eqn:astAlg}
\astAlg \,:\, \astOp(\MM,J,j)^\op ~\longrightarrow ~\Cat \quad.
\end{flalign}
Given any $\astOp(\MM,J,j)$-morphism $(f,\phi): (\CCC,\O) \to (\DDD,\P)$,  
we denote the corresponding functor simply by
\begin{flalign}\label{eqn:astAlgpullback}
(f,\phi)^\ast := \astAlg(f,\phi) \,:\, \astAlg(\P) ~\longrightarrow~ \astAlg(\O) \quad.
\end{flalign}
Concretely, it is given by evaluating the $2$-functor 
$\astObj: \ICat \to \Cat$ given in \eqref{eqn:astObj2fun} 
on the involutive functor \eqref{eqn:Algpullbackinv}. 

\begin{rem}
Recalling Proposition \ref{propo:OpastObj},
there exists an isomorphism
\begin{flalign}
\astOp_\CCC(\MM,J,j) ~\cong ~\Op_\CCC \big( \astObj(\MM,J,j) \big)
\end{flalign}
between the category of colored $\ast$-operads with values in
$(\MM,J,j)$ and the category of ordinary colored operads
with values in $\astObj(\MM,J,j)$. This isomorphism
clearly extends to the categories of colored ($\ast$-)operads
with varying colors. As a consequence, there exists a second option
for assigning categories of $\ast$-algebras to colored $\ast$-operads,
which is given by the lower path in the diagram
\begin{flalign}\label{eqn:diagramastAlg}
\xymatrix{
\ar[rd]_-{\cong~~}\astOp(\MM,J,j)^\op  \ar[rr]^-{\astAlg} &\ar@{=>}[d]^-{\cong}& \Cat\\
&\Op \big( \astObj(\MM,J,j) \big)^\op\ar[ru]_-{~~\Alg}&
}
\end{flalign}
where $\Alg$ denotes the functor given in \eqref{eqn:Alg}.
Similarly to  \cite[Proposition 3]{Jacobs}, one can prove  that
the diagram \eqref{eqn:diagramastAlg} commutes up to a natural isomorphism,
hence the second option for assigning the categories of $\ast$-algebras is 
equivalent to our original definition in \eqref{eqn:astAlg}.
\sk

We would like to emphasize that the main reason
why the diagram in \eqref{eqn:diagramastAlg}
commutes is that the conditions (1-3) in Remark \ref{rem:astAlgexplicit}
admit two equivalent interpretations: 
The first option is to regard $(A,\alpha) \in \Alg(\O)$ 
as an algebra over the $\CCC$-colored operad $\O$
and $\ast_A: (A,\alpha) \to J_\O (A,\alpha)$ 
as an $\Alg(\O)$-morphism. One observes that 
$\big( \ast_A: (A,\alpha) \to J_\O (A,\alpha) \big) 
\in \astObj \big( \Alg(\O), J_\O, j_\O \big)$ is a $\ast$-object 
in the involutive category $\big( \Alg(\O), J_\O, j_\O \big)$, 
which recovers our original Definition \ref{def:astAlg} and hence the upper path
in the diagram \eqref{eqn:diagramastAlg}. 
The second option is to regard 
$(\ast_A: A \to J_\ast A) \in \astObj(\MM,J,j)^\CCC$ 
as a $\CCC$-colored object in $\astObj(\MM,J,j)$ 
and $\alpha: (\ast: \O \to J_\ast \O) \circ (\ast_A: A \to J_\ast A) 
\to (\ast_A: A \to J_\ast A)$ as a $\astObj(\MM,J,j)^\CCC$-morphism.
One observes that this defines an algebra
over $\O$, regarded as an object in $\Op_\CCC(\astObj(\MM,J,j))$,
which recovers the lower path in the diagram \eqref{eqn:diagramastAlg}.
\end{rem}

We conclude this section by noticing that \eqref{eqn:Algpullbackinv} 
equips the right adjoint functor $(f,\phi)^\ast: \Alg(\P) \to \Alg(\O)$ 
of the adjunction \eqref{eqn:changeOp} with an involutive structure. 
Hence, applying Proposition \ref{propo:invadj}, we obtain 
a canonical involutive structure 
$\lambda_{(f,\phi)}: (f,\phi)_!\, J_\O \to J_\P\, (f,\phi)_!$
on the left adjoint functor $(f,\phi)_!: \Alg(\O) \to \Alg(\P)$ 
together with an involutive adjunction 
\begin{flalign}\label{eqn:changeastOp} 
\xymatrix{
\big( (f,\phi)_!, \lambda_{(f,\phi)} \big)\,:\, \big( \Alg(\O), J_\O, j_\O \big) 
~\ar@<0.5ex>[r] & \ar@<0.5ex>[l]~ \big( \Alg(\P), J_\P, j_\P \big) \,:\, 
\big( (f,\phi)^\ast, \id_{(f,\phi)^\ast\, J_\P} \big)
}\quad.
\end{flalign}
Because $2$-functors preserve adjunctions, 
we may apply the $2$-functor $\astObj : \ICat\to\Cat$
to the involutive adjunction \eqref{eqn:changeastOp} 
in order to obtain an adjunction
\begin{flalign}\label{eqn:changeastOpastAlg} 
\xymatrix{
(f,\phi)_!\,:\, \astAlg(\O) ~\ar@<0.5ex>[r] 
& \ar@<0.5ex>[l]~ \astAlg(\P) \,:\, (f,\phi)^\ast
}
\end{flalign}
between the categories of $\ast$-algebras.
Summing up, we have proven
\begin{theo}\label{theo:adjunctionastAlg}
Associated to every $\astOp(\MM,J,j)$-morphism 
$(f,\phi): ( \CCC, \O ) \to ( \DDD, \P)$, there is
an involutive adjunction \eqref{eqn:changeastOp}
between the involutive categories of algebras
and an adjunction \eqref{eqn:changeastOpastAlg} 
between the categories of $\ast$-algebras.
\end{theo}


\section{\label{sec:QFTs}Algebraic quantum field theory $\ast$-operads}
As an application of the concepts and techniques developed in this paper,
we study the family of colored operads arising in algebraic quantum field theory
\cite{BeniniSchenkelWoike} within the setting of involutive category theory. 
The main motivation for promoting these colored operads to colored $\ast$-operads
is due to quantum physics: A quantum mechanical system is described not only by an associative 
and unital algebra over $\bbC$, but rather by an associative and unital $\ast$-algebra $A$ over $\bbC$.  
Here the relevant type of $\ast$-algebras is the reversing one, i.e.\ $(a\,b)^\ast = b^\ast\,a^\ast$.
The additional structure given by the complex anti-linear $\ast$-involution is essential 
for quantum physics: It enters the GNS construction that is crucial to 
recover the usual probabilistic interpretation of quantum theory in terms of Hilbert spaces.
\sk

Throughout this section we let $(\MM,J,j)$ be any cocomplete involutive closed symmetric monoidal category.
In traditional quantum field theory, one  would choose the example given by 
complex vector spaces $(\Vec_\bbC,\ovr{(-)},\id_{\ID_{\Vec_\bbC}})$, see Examples 
\ref{ex:vec}, \ref{ex:vec2} and \ref{ex:vec3} for details. More modern approaches to
quantum {\em gauge} theories, however, have lead to the concept of \emph{homotopical quantum field theory} 
and crucially rely on using different and richer target categories,
such as chain complexes and other monoidal model categories, see e.g.\
 \cite{BeniniSchenkelSzabo,BeniniSchenkel,BeniniSchenkelWoike,BeniniSchenkelWoikehomotopy,YauQFT}
for algebraic quantum field theory and also \cite{CostelloGwilliam} for similar developments in factorization algebras.
Hence, it is justified to present our constructions  with this high level of generality.
\sk

Let us provide a very brief review of the algebraic quantum field theory operads
constructed in \cite{BeniniSchenkelWoike}. We refer to this paper for more 
details and the physical motivations. 
\begin{defi}
An {\em orthogonality relation} on a small category $\CC$ is a subset
${\perp} \subseteq \mathrm{Mor}\,\CC{}_\mathrm{t} \times_{\mathrm{t}} \mathrm{Mor}\,\CC$
of the set of pairs of $\CC$-morphisms with coinciding target
that is symmetric, i.e.\ $(f_1,f_2)\in{\perp}$ implies $(f_2,f_1)\in{\perp}$,
and stable under post- and pre-composition, i.e.\ $(f_1,f_2)\in{\perp}$
implies $(g f_1,g f_2)\in {\perp}$ and $(f_1 h_1,f_2 h_2)\in{\perp}$ for all
composable $\CC$-morphisms $g$, $h_1$ and $h_2$. 
We call elements $(f_1,f_2)\in \perp$ {\em orthogonal pairs} and also write $f_1\perp f_2$.
A pair $(\CC,\perp)$ consisting of a small category $\CC$ and an orthogonality relation $\perp$
on $\CC$ is called an {\em orthogonal category}.
\end{defi}

\begin{ex}\label{ex:terminalOCAT}
On the terminal category $\CC=\{\bullet\}$ there exist precisely two different orthogonality relations, 
namely $\perp = \emptyset$ and $\perp=\{(\id_\bullet,\id_\bullet)\}$.
The corresponding orthogonal categories
$(\{\bullet\},\emptyset)$ and $(\{\ast\},\{(\id_\bullet,\id_\bullet)\})$ 
will be used below to illustrate our constructions for the simplest possible examples.
\end{ex}

\begin{ex}\label{ex:LocOCAT}
The following is the prime example of an orthogonal category, see 
e.g.\ \cite{Brunetti,BeniniSchenkelWoike} for the details. Let $\Loc$
be the category of globally hyperbolic Lorentzian manifolds (of a fixed dimension $\geq 2$) with morphisms given
by causally convex and open isometric embeddings.
Two morphisms $f_1 : M_1\to M$ and $f_2 : M_2\to M$ to a common Lorentzian manifold $M$
are declared to be orthogonal, $f_1\perp f_2$, if and only if their images are causally disjoint subsets
of $M$, i.e.\ there exists no causal curve connecting $f_1(M_1)$ and $f_2(M_2)$.
The resulting orthogonal category $(\Loc,\perp)$ describes the physical concept of 
{\em spacetimes} (in the sense of Einstein's general relativity)
and their causal relations. It provides the foundation for formulating
locally covariant algebraic quantum field theory \cite{Brunetti}.
\sk

Another related example is obtained by the following construction:
Choosing any globally hyperbolic Lorentzian manifold $M\in\Loc$, consider the over category
$\Loc/M$ together with the forgetful functor $U : \Loc/M\to\Loc$.
The orthogonality relation $\perp$ on $\Loc$ pulls back under $U$ to an orthogonality relation
$\perp_M$ on $\Loc/M$. Explicitly, two morphisms $g_1$ and $g_2$ to a common target in $\Loc/M$
are orthogonal with respect to $\perp_M$ if and only if $U(g_1)\perp U(g_2)$ in $(\Loc,\perp)$. The resulting
orthogonal category $(\Loc/M,\perp_M)$ describes causally convex open subsets
of the fixed globally hyperbolic Lorentzian manifold $M$ (interpreted physically as the universe) 
and their causal relations.
It provides the foundation for formulating Haag-Kastler type algebraic quantum field theories \cite{HaagKastler}.
\end{ex}

Let $(\CC,\perp)$ be an orthogonal category and denote by $\CC_0$ the set of objects of $\CC$.
To define the algebraic quantum field theory operad associated to $(\CC,\perp)$
it is convenient to introduce the following notations:
Given $\und{c}=(c_1,\dots,c_n)\in \Sigma_{\CC_0}$ and $t\in\CC$,
we denote by $\CC(\und{c},t) := \prod_{i=1}^n \CC(c_i,t)$ the product of $\Hom$-sets.
Its elements will be denoted by symbols like $\und{f} = (f_1,\dots,f_n)\in \CC(\und{c},t)$.
The following definition is due to \cite{BeniniSchenkelWoike}.
\begin{defi}\label{defi:AQFToperad}
Let $(\CC,\perp)$ be an orthogonal category.
The {\em algebraic quantum field theory operad} of type $(\CC,\perp)$ 
with values in $\MM$ is the $\CC_0$-colored operad $\O_{(\CC,\perp)}\in \Op_{\CC_0}(\MM)$ 
defined as follows:
\begin{itemize}
\item[(a)] For any $(\und{c},t)\in \Sigma_{\CC_0}\times\CC_0$, we set
\begin{flalign}
\O_{(\CC,\perp)}\big(\substack{t\\\und{c}}\big) \,:=\, \big(\Sigma_{\vert\und{c}\vert} \times \CC(\und{c},t) \big)\big/ {\sim_\perp}  \otimes I  \in\MM \quad,
\end{flalign}
where the equivalence relation is as follows: 
$(\sigma,\und{f}) \sim_{\perp} (\sigma^\prime,\und{f}^\prime)$ if and only if
(1)~$\und{f} = \und{f}^\prime$ and (2)~the right permutation
$\sigma{\sigma^{\prime}}^{-1} : \und{f}\sigma^{-1} \to \und{f}{\sigma^{\prime}}^{-1}$
is generated by transpositions of adjacent orthogonal pairs.

\item[(b)] For any $\Sigma_{\CC_0}\times \CC_0$-morphism $\sigma^\prime : (\und{c},t)\to
(\und{c}\sigma^\prime,t)$, we set
\begin{flalign}
\O_{(\CC,\perp)}(\sigma^\prime)\, :\, \O_{(\CC,\perp)}\big(\substack{t\\\und{c}}\big) ~\longrightarrow~\O_{(\CC,\perp)}\big(\substack{t\\\und{c}\sigma^\prime}\big)
\end{flalign}
to be the $\MM$-morphism induced by the map of sets 
$[\sigma,\und{f}]\mapsto [\sigma\sigma^\prime, \und{f}\sigma^\prime]$ via functoriality 
of the $\Set$-tensoring.

\item[(c)] The operadic composition is determined by the $\MM$-morphisms 
\begin{subequations}
\begin{flalign}
\gamma \,:\, \O_{(\CC,\perp)}\big(\substack{t\\\und{a}}\big)\otimes\Motimes_{i=1}^m 
\O_{(\CC,\perp)}\big(\substack{a_i\\\und{b}_i}\big) ~\longrightarrow~\O_{(\CC,\perp)}\big(\substack{t\\\und{b}_1\otimes\cdots\otimes \und{b}_m}\big)
\end{flalign}
induced by the maps of sets
\begin{flalign}
[\sigma,\und{f}] \otimes\Motimes_{i=1}^m [\sigma_i,\und{g}_i]~\longmapsto~
\big[ \sigma(\sigma_1,\dots,\sigma_m), \und{f}(\und{g}_1,\dots,\und{g}_m)\big]
\end{flalign}
\end{subequations}
via functoriality of the $\Set$-tensoring. Here $\sigma(\sigma_1,\dots,\sigma_m) = \sigma\langle \vert\und{b}_{\sigma^{-1}(1)}\vert,\dots,\vert\und{b}_{\sigma^{-1}(m)}\vert\rangle~(\sigma_1\oplus\cdots\oplus \sigma_{m})$
denotes the group multiplication in $\Sigma_{\vert \und{b}_1\vert + \cdots + \vert \und{b}_m\vert}$ 
of the corresponding block permutation and block sum permutation,
and $\und{f}(\und{g}_1,\dots,\und{g}_m) = (f_1 \, g_{11},\dots,f_m \, g_{m\vert \und{b}_m\vert})$ is given by 
composition of $\CC$-morphisms.

\item[(d)] The operadic unit is determined by the $\MM$-morphisms 
\begin{flalign}
\oone\,:\, I~\longrightarrow~\O_{(\CC,\perp)}\big(\substack{t\\t}\big)
\end{flalign}
induced by the maps of sets
$\bullet \mapsto (e,\id_t)$, where $e\in \Sigma_1$ is the group unit,
via functoriality of the $\Set$-tensoring.
\end{itemize}
\end{defi}

The following results are proven in \cite{BeniniSchenkelWoike}.
\begin{theo}\label{theo:AlgAQFTfunctors}
For any orthogonal category $(\CC,\perp)$, Definition \ref{defi:AQFToperad} defines
a $\CC_0$-colored operad $\O_{(\CC,\perp)}\in\Op_{\CC_0}(\MM)$. Furthermore, 
there exists an isomorphism
\begin{flalign}
\Alg(\O_{(\CC,\perp)})~\cong~\Mon(\MM)^{(\CC,\perp)}
\end{flalign}
between the category of $\O_{(\CC,\perp)}$-algebras and the category of $\perp$-commutative
functors from $\CC$ to $\Mon(\MM)$. Concretely, the latter is the full subcategory
of the functor category $\Mon(\MM)^\CC$ consisting of all functors $\mathfrak{A} : \CC\to \Mon(\MM)$
for which the diagrams
\begin{flalign}\label{eqn:perpcommutativity}
\xymatrix{
\ar[d]_-{\mathfrak{A}(f_1)\otimes \mathfrak{A}(f_2)}\mathfrak{A}(c_1)\otimes \mathfrak{A}(c_2) \ar[rr]^-{\mathfrak{A}(f_1)\otimes \mathfrak{A}(f_2)} && \mathfrak{A}(c)\otimes \mathfrak{A}(c)\ar[d]^-{\mu_c^\op}\\
 \mathfrak{A}(c)\otimes \mathfrak{A}(c) \ar[rr]_-{\mu_c} &&\mathfrak{A}(c)
}
\end{flalign}
in $\MM$ commute, for all orthogonal pairs $f_1\perp f_2$. Here $\mu_c$ (respectively $\mu_c^\op$) denotes 
the (opposite) multiplication in the monoid $\mathfrak{A}(c)$.
\end{theo}

\begin{ex}\label{ex:MonandCMon}
Recalling the orthogonal categories from Example \ref{ex:terminalOCAT},
one easily observes that the family of colored operads in Definition \ref{defi:AQFToperad}
includes the associative operad and the commutative operad as very special cases.
Concretely, $\O_{(\{\bullet\},\emptyset)} = \mathsf{As}$ is the associative operad
and hence $\Alg(\O_{(\{\bullet\},\emptyset)})\cong \Mon(\MM)$ is the category of monoids in $\MM$. 
Similarly, $\O_{(\{\bullet\},\{(\id_\bullet,\id_\bullet)\})} = \mathsf{Com}$ is the commutative operad
and hence $\Alg(\O_{(\{\bullet\},\{(\id_\bullet,\id_\bullet)\})})\cong \CMon(\MM) $ is the category
of commutative monoids in $\MM$. 
\end{ex}

\begin{ex}\label{ex:AQFTs}
More interestingly, the colored operad $\O_{(\Loc,\perp)}$ associated to the orthogonal
category $(\Loc,\perp)$ from Example \ref{ex:LocOCAT} describes locally covariant algebraic
quantum field theories in the sense of \cite{Brunetti}, i.e.\ $\Alg(\O_{(\Loc,\perp)})\cong \mathbf{QFT}(\Loc)$
is the category of such theories. The $\perp$-commutativity
property in \eqref{eqn:perpcommutativity} formalizes the Einstein causality axiom,
which states that observables localized in causally disjoint subsets commute with each other.
The colored operad $\O_{(\Loc/M,\perp_M)}$ associated to the over category
describes Haag-Kastler type algebraic quantum field theories \cite{HaagKastler}
on the fixed Lorentzian manifold $M\in\Loc$, i.e.\ $\Alg(\O_{(\Loc/M,\perp_M)})\cong \mathbf{QFT}(M)$
is the category of such theories. By Theorem \ref{theo:AlgAQFTfunctors},
these are characterized as pre-cosheaves of monoids on $M$ satisfying Einstein causality.
\end{ex}

We will now endow $\O_{(\CC,\perp)}\in \Op_{\CC_0}(\MM)$
with the structure of a colored $\ast$-operad. According to Remark \ref{rem:astOp},
this amounts to equipping the symmetric sequence underlying $\O_{(\CC,\perp)}$
with the structure of a $\ast$-object in the involutive monoidal category $(\SymSeq_{\CC_0}(\MM),J_\ast,j_\ast)$
that is compatible with the operadic compositions and units.
Let us define a $\SymSeq_{\CC_0}(\MM)$-morphism $\ast: \O_{(\CC,\perp)} \to J_\ast \O_{(\CC,\perp)}$
by setting, for all $(\und{c},t)\in\Sigma_{\CC_0}\times\CC_0$,
\begin{flalign}\label{eqn:AQFTast}
\xymatrix{
\ar@{=}[d]\O_{(\CC,\perp)}\big(\substack{t\\\und{c}}\big) \ar[rr]^-{\ast} && J \O_{(\CC,\perp)}\big(\substack{t\\\und{c}}\big)\ar[d]^-{\cong}\\
\big(\Sigma_{\vert\und{c}\vert} \times \CC(\und{c},t)\big)\big/ {\sim_\perp}  \otimes I \ar[rr]_-{\rho_{\vert\und{c}\vert} \otimes J_0}&&
\big(\Sigma_{\vert\und{c}\vert} \times \CC(\und{c},t)\big)\big/ {\sim_\perp}  \otimes JI 
}
\end{flalign}
to be the $\MM$-morphism induced by the map of sets $\rho_{\vert\und{c}\vert} : 
[\sigma,\und{f}]\mapsto [\rho_{\vert\und{c}\vert}\sigma,\und{f}]$, where $\rho_{\vert\und{c}\vert}\in\Sigma_{\vert\und{c}\vert}$ 
is the order-reversal permutation from Example \ref{ex:Cprofiles}, and the $\MM$-morphism $J_0 : I\to JI$. 
(For the right vertical arrow recall that $J$ is self-adjoint, hence it preserves the $\Set$-tensoring.)
Evidently, \eqref{eqn:AQFTast} is equivariant with respect to the action of permutations
given in Definition \ref{defi:AQFToperad}~(b), hence it defines a $\SymSeq_{\CC_0}(\MM)$-morphism.
It is, moreover, straightforward to verify that $(\ast: \O_{(\CC,\perp)} \to 
J_\ast \O_{(\CC,\perp)})\in\astObj(\SymSeq_{\CC_0}(\MM),J_\ast,j_\ast)$ is a $\ast$-object
by using that $\rho_{\vert\und{c}\vert}^2 =e$ is the identity permutation 
and that $j : \ID_\MM \to J^2$ is by hypothesis a monoidal natural transformation.
\begin{propo}\label{propo:AQFTastOp}
Endowing the colored operad $(\O_{(\CC,\perp)},\gamma,\oone)\in \Op_{\CC_0}(\MM)$
from Definition \ref{defi:AQFToperad} with the $\ast$-involution $\ast: \O_{(\CC,\perp)} \to J_\ast \O_{(\CC,\perp)}$
defined in \eqref{eqn:AQFTast} yields a colored $\ast$-operad
\begin{flalign}
\big(\O_{(\CC,\perp)},\gamma,\oone,\ast\big)\in\astOp_{\CC_0}(\MM,J,j)\quad.
\end{flalign}
\end{propo}
\begin{proof}
It remains to check the compatibility conditions in Remark \ref{rem:astOp}~(3).
This is a straightforward calculation using standard permutation group properties.
\end{proof}

Let us now study the $\ast$-algebras over the colored $\ast$-operad
$\O_{(\CC,\perp)}\in\astOp_{\CC_0}(\MM,J,j)$ defined in Proposition \ref{propo:AQFTastOp}.
Using the explicit description explained in Remark \ref{rem:astAlgexplicit},
these are triples $(A,\alpha,\ast_A)$ consisting of an algebra $(A,\alpha)$
over $\O_{(\CC,\perp)}$ together with a compatible $\ast$-involution
$\ast_A : A \to J_\ast A$. Using Theorem \ref{theo:AlgAQFTfunctors}
to identify $(A,\alpha)$ with a $\perp$-commutative
functor $\mathfrak{A} : \CC\to \Mon(\MM)$, the
$\ast$-involution $\ast_A : A \to J_\ast A$ is identified with
a family of $\MM$-morphisms
\begin{flalign}
\ast_c \,:\, \mathfrak{A}(c)~\longrightarrow J \mathfrak{A}(c)\quad,
\end{flalign}
for all $c\in\CC$. As a consequence of Remark \ref{rem:astAlgexplicit}~(3), 
such family has to satisfy the following basic conditions:
\begin{itemize}
\item[(1)] {\em Compatibility with monoid structure:} For all $c\in \CC$, 
\begin{subequations}\label{eqn:reversingastMonoid}
\begin{flalign}\label{eqn:reversingastMonoid1}
\xymatrix@C=3.5em{
\ar[d]_-{\mu_c}\mathfrak{A}(c) \otimes \mathfrak{A}(c) \ar[r]^-{\ast_c\otimes\ast_c} & J\mathfrak{A}(c)\otimes J\mathfrak{A}(c) \ar[r]^-{{J_2}_{\mathfrak{A}(c),\mathfrak{A}(c)}} & J\big(\mathfrak{A}(c)\otimes \mathfrak{A}(c)\big) \ar[d]^-{J\mu^\op_c}\\
\mathfrak{A}(c) \ar[rr]_-{\ast_c}&& J\mathfrak{A}(c)
}
\end{flalign}
where $\mu_c$ (respectively $\mu_c^\op$) is the (opposite) multiplication on $\mathfrak{A}(c)\in\Mon(\MM)$,
and
\begin{flalign}
\xymatrix@C=3.5em{
\ar[d]_-{\eta_c} I \ar[r]^-{J_0} & J I \ar[d]^-{J \eta_c}\\
\mathfrak{A}(c) \ar[r]_-{\ast_c} & J\mathfrak{A}(c)
}
\end{flalign}
\end{subequations}
where $\eta_c$ is the unit on $\mathfrak{A}(c)\in\Mon(\MM)$.

\item[(2)] {\em Compatibility with functor structure:} For all $\CC$-morphisms $f : c\to c^\prime$,
\begin{flalign}\label{eqn:functorialityastMonoid}
\xymatrix@C=3.5em{
\ar[d]_-{\mathfrak{A}(f)} \mathfrak{A}(c) \ar[r]^-{\ast_c} & J\mathfrak{A}(c)\ar[d]^-{J\mathfrak{A}(f)}\\
\mathfrak{A}(c^\prime) \ar[r]_-{\ast_{c^\prime}} & J\mathfrak{A}(c^\prime)
}
\end{flalign}
\end{itemize}

\begin{ex}\label{ex:reversal}
To illustrate the behavior of these $\ast$-involutions, consider the orthogonal category
$(\{\bullet\} ,\emptyset)$ from Examples \ref{ex:terminalOCAT} and \ref{ex:MonandCMon}.
Then $\O_{(\{\bullet\},\emptyset)} = \mathsf{As}$ is the associative operad and Proposition \ref{propo:AQFTastOp}
defines a $\ast$-operad structure on it. For later convenience,
let us denote the corresponding category of $\ast$-algebras by
\begin{flalign}\label{eqn:astMonrev}
\astMon_{\mathrm{rev}}(\MM,J,j)~:=~\astAlg(\O_{(\{\bullet\},\emptyset)})\quad.
\end{flalign}
Using our concrete description from above, an object in this category
is a quadruple $(A,\mu,\eta,\ast)$ consisting of a monoid $(A,\mu,\eta)\in \Mon(\MM)$
together with a $\ast$-involution $\ast : A\to J A$ satisfying the compatibility
conditions in \eqref{eqn:reversingastMonoid}. (The conditions in \eqref{eqn:functorialityastMonoid} 
are vacuous because we consider the discrete category $\{\bullet\}$ in this example.)
Comparing these structures to $\ast$-monoids, cf.\ Remark \ref{rem:astmonoids},
we observe that they are very similar, up to the appearance of
the opposite multiplication in \eqref{eqn:reversingastMonoid}. 
This order-reversal of the multiplication under $\ast$-involution, which results from
our $\ast$-operad structure \eqref{eqn:AQFTast},
motivates our notation $\astMon_{\mathrm{rev}}(\MM,J,j)$.
\sk

As a very concrete example, and referring back to Example \ref{ex:Vecexplicitmonoids}, 
let us consider the involutive symmetric monoidal category 
$(\Vec_\bbC,\ovr{(-)},\id_{\ID_{\Vec_\bbC}})$ 
from Example \ref{ex:vec3}. In this case \eqref{eqn:astMonrev}
describes the category of order-reversing associative
and unital $\ast$-algebras over $\bbC$, i.e.\ $(a\,b)^\ast = b^\ast\,a^\ast$,
which is of major relevance for (traditional) quantum physics.
\end{ex}

\begin{rem}\label{rem:revMon}
We would like to mention that \eqref{eqn:AQFTast} is not the only
possible $\ast$-involution on the colored operad $\O_{(\CC,\perp)}$. For example,
we could replace the order-reversal permutations $\rho_{\vert\und{c}\vert}$ in
\eqref{eqn:AQFTast} by the identity permutations $e$. This would define
another colored $\ast$-operad structure on $\O_{(\CC,\perp)}$ that differs
from our choice above. The $\ast$-algebras for this alternative choice 
do not describe order-reversing $\ast$-involutions. In particular, 
$\ast$-algebras over $\O_{(\{\bullet\},\emptyset)}=\mathsf{As}$
for this choice of $\ast$-involution are non-reversing $\ast$-monoids as in 
Remark \ref{rem:astmonoids}. Hence, our general framework
for (colored) $\ast$-operads is sufficiently flexible to capture both
reversing and non-reversing $\ast$-involutions on monoids,
which correspond to different choices of $\ast$-operad structures on 
the same underlying operad $\O_{(\{\bullet\},\emptyset)}=\mathsf{As}$.
\end{rem}

In general, we have the following explicit characterization of $\ast$-algebras
over the colored $\ast$-operad
$\O_{(\CC,\perp)}\in\astOp_{\CC_0}(\MM,J,j)$ defined in Proposition \ref{propo:AQFTastOp}.
\begin{propo}\label{propo:astAQFT}
For any orthogonal category $(\CC,\perp)$, there exists an isomorphism
\begin{flalign}
\astAlg(\O_{(\CC,\perp)}) ~\cong~ \astMon_{\mathrm{rev}}(\MM,J,j)^{(\CC,\perp)}
\end{flalign}
between the category of $\ast$-algebras over $\O_{(\CC,\perp)}$ and the
category of $\perp$-commutative functors from $\CC$ to
the category of order-reversing $\ast$-monoids in $(\MM,J,j)$, cf.\ Example \ref{ex:reversal}.
\end{propo}
\begin{proof}
This is an immediate consequence of Theorem \ref{theo:AlgAQFTfunctors}
together with \eqref{eqn:reversingastMonoid} and \eqref{eqn:functorialityastMonoid}.
\end{proof}

\begin{ex}
Applying this result to Example \ref{ex:AQFTs}, we observe that
the category of $\ast$-algebras over the colored $\ast$-operad
$\O_{(\Loc,\perp)}$ is the category of locally covariant algebraic quantum field
theories endowed with $\ast$-involutions, $\astAlg(\O_{(\Loc,\perp)})\cong \ast\text{-}\mathbf{QFT}(\Loc)$.
The order-reversing nature of the $\ast$-involutions is precisely what is needed in quantum 
physics \cite{HaagKastler,Brunetti}.
In complete analogy, the category of $\ast$-algebras over the colored $\ast$-operad
$\O_{(\Loc/M,\perp_M)}$ is the category of Haag-Kastler type algebraic quantum field
theories on the Lorentzian manifold $M$ endowed with $\ast$-involutions, 
$\astAlg(\O_{(\Loc/M,\perp_M)})\cong \ast\text{-}\mathbf{QFT}(M)$.
\end{ex}

We conclude this section with some further remarks on
constructions and results that are of interest in quantum field theory.

\paragraph*{Change of orthogonal category adjunctions:}
The assignment $(\CC,\perp)\mapsto \O_{(\CC,\perp)}$ of our colored
$\ast$-operads is functorial
\begin{flalign}
\O \,:\, \OCat ~\longrightarrow ~ \astOp(\MM,J,j)
\end{flalign}
on the category of orthogonal categories, where a morphism
$F : (\CC,\perp)\to (\CC^\prime,\perp^\prime)$ is a functor
preserving the orthogonality relations in the sense of 
$F(\perp)\subseteq {\perp}^\prime$.
Together with Theorem \ref{theo:adjunctionastAlg},
this implies
\begin{cor}\label{cor:astAlgadjunction}
Associated to every $\OCat$-morphism $F : (\CC,\perp)\to (\CC^\prime,\perp^\prime)$
there is an adjunction
\begin{flalign}
\xymatrix{
{\O_F}_!\,:\, \astAlg(\O_{(\CC,\perp)}) ~\ar@<0.5ex>[r] 
& \ar@<0.5ex>[l]~ \astAlg(\O_{(\CC^\prime,\perp^\prime)}) \,:\, {\O_F}^\ast
}\quad.
\end{flalign}
\end{cor}
\begin{rem}
Such adjunctions have plenty of quantum field theoretic applications,
see e.g.\ \cite{BeniniSchenkelWoike} and also \cite{BeniniDappiaggiSchenkel}
for concrete examples. The results of this section show that these adjunctions are also available in
the involutive setting, which is crucial to describe the order-reversing associative and unital $\ast$-algebras
appearing in quantum field theory.
\end{rem}

\paragraph*{States and the GNS construction:}
Building on the results in \cite{Jacobs}, we shall briefly explain
the GNS construction for order-reversing $\ast$-monoids
and $\ast$-algebraic quantum field theories with values in an arbitrary cocomplete
involutive closed symmetric monoidal category $(\MM,J,j)$.
This requires some preparatory definitions and terminology.
\begin{defi}\label{def:staterep}
\begin{itemize}
\item[(a)] A {\em state} on an order-reversing $\ast$-monoid
$(A,\mu,\eta,\ast)\in \astMon_{\mathrm{rev}}(\MM,J,j)$
is a $\astObj(\MM,J,j)$-morphism 
\begin{flalign}
\omega \,: \, \big(\ast : A\to JA\big) ~\longrightarrow ~\big(J_0 : I\to JI\big)\quad.
\end{flalign}
To simplify notation, we shall write $\omega : A\to I$ for a state on $A$.

\item[(b)] Given any object $V\in\MM$, define the following $\ast$-object structure 
\begin{flalign}
\xymatrix{
\ar[d]_-{\tau_{JV,V}} JV\otimes V \ar[rr]^-{\ast_{JV\otimes V}} && J(JV\otimes V)\\
V\otimes JV \ar[rr]_-{j_V\otimes \id}&& J^2V\otimes JV\ar[u]_-{{J_2}_{JV,V}}
}
\end{flalign}
on $JV\otimes V$. An {\em inner product space} in $(\MM,J,j)$ is a pair $(V,\langle\cdot,\cdot\rangle)$
consisting of an object $V\in\MM$ and a $\astObj(\MM,J,j)$-morphism
\begin{flalign}
\langle\cdot,\cdot\rangle \,:\, \big(\ast_{JV\otimes V} : JV\otimes V \to J(JV\otimes V)\big)~\longrightarrow~ \big(J_0:I\to JI\big)\quad.
\end{flalign}
To simplify notation, we shall write $\langle\cdot,\cdot\rangle: JV\otimes V\to I$
for an inner product.

\item[(c)] A {\em $\ast$-representation} of an order-reversing $\ast$-monoid
$(A,\mu,\eta,\ast)\in \astMon_{\mathrm{rev}}(\MM,J,j)$ on an inner product space
$(V,\langle\cdot,\cdot\rangle)$ in $(\MM,J,j)$ is  a left $(A,\mu,\eta)$-module structure
$\ell : A\otimes V\to V$ on $V$ that is compatible with the inner product, i.e.\ the diagram
\begin{flalign}
\xymatrix@C=4em{
\ar[d]_-{\tau_{JV,A}\otimes \id} JV\otimes A\otimes V \ar[rrr]^-{\id\otimes\ell}&&&JV\otimes V\ar[dd]^-{\langle\cdot,\cdot\rangle}\\
\ar[d]_-{\ast\otimes\id\otimes\id} A\otimes JV\otimes V&&&\\
JA\otimes JV\otimes V\ar[r]_-{{J_2}_{A,V}\otimes \id} & J(A\otimes V)\otimes V\ar[r]_-{J\ell \otimes \id} &JV\otimes V \ar[r]_-{\langle\cdot,\cdot\rangle} & I
}
\end{flalign}
in $\MM$ commutes.
\end{itemize}
\end{defi}

\begin{rem}
Notice that there is no concept of {\em positivity} for
a state $\omega : A\to I$ or an inner product $\langle\cdot,\cdot\rangle : JV\otimes V\to I$
in an arbitrary involutive symmetric monoidal category $(\MM,J,j)$. That is why Definition \ref{def:staterep}
does not take this property into account. For certain examples, e.g.\ the involutive
category $(\Vec_\bbC,\overline{(-)},\id_{\ID_{\Vec_{\bbC}}})$ of complex vector spaces,
one may select positive states and positive inner products by imposing additional
conditions on the states and inner product spaces in the sense of Definition \ref{def:staterep}.
Concretely, a state $\omega : A\to \bbC$ is positive if $\omega(a^\ast\,a)\geq 0$,
for all $a\in A$, and an inner product $\langle\cdot,\cdot\rangle : \overline{V}\otimes V\to \bbC$
is positive if $\langle v,v\rangle \geq 0$, for all $v\in V$.
\end{rem}

The GNS construction for order-reversing $\ast$-monoids 
in $(\MM,J,j)$ is as follows.
\begin{propo}\label{propo:GNS}
Let $\omega : A\to I$ be a state on an order-reversing $\ast$-monoid 
$(A,\mu,\eta,\ast)\in\astMon_{\mathrm{rev}}(\MM,J,j)$. Then
\begin{flalign}
\xymatrix{
JA\otimes A \ar[rr]^-{\langle\cdot,\cdot\rangle} && I\\
\ar[u]_-{\cong}^-{\ast\otimes \id} A\otimes A \ar[rr]_-{\mu}&& A \ar[u]_-{\omega}
}
\end{flalign}
defines an inner product space structure on the underlying object $A\in\MM$. Moreover,
$\ell = \mu : A\otimes A\to A$ defines a $\ast$-representation of
$(A,\mu,\eta,\ast)$ on $(A,\langle\cdot,\cdot\rangle)$.
\end{propo}
\begin{proof}
This is an elementary diagram chase using in particular the property
\eqref{eqn:reversingastMonoid} for order-reversing $\ast$-monoids.
\end{proof}

\begin{ex}
This concept of states and $\ast$-representations generalizes immediately
to $\ast$-algebraic quantum field theory. Let $(\CC,\perp)$ be any
orthogonal category and $\AAA\in \astAlg(\O_{(\CC,\perp)})$
a $\ast$-algebra over the corresponding 
algebraic quantum field theory $\ast$-operad.
By Proposition \ref{propo:astAQFT}, we can describe $\AAA$ as
a $\perp$-commutative functor $\AAA : \CC\to\astMon_{\mathrm{rev}}(\MM,J,j)$
with values in the category of order-reversing $\ast$-monoids.
The usual concept of states considered in algebraic quantum field theory
is point-wise, see e.g.\ \cite{HaagKastler,Brunetti}. Concretely,
we define a state on $\AAA$ to be a family $\omega_c : \AAA(c)\to I$ of
states in the sense of Definition \ref{def:staterep},
for all objects $c\in\CC$, such that $\omega_{c^\prime}\,\AAA(f) = \omega_c$, 
for every $\CC$-morphism $f:c\to c^\prime$. Applying the GNS construction from
Proposition \ref{propo:GNS}, we obtain a family of inner product spaces
$(\AAA(c),\langle\cdot,\cdot\rangle_c)$ and a family of
$\ast$-representations that are functorial in $c$. In case $\CC$ has a terminal object $t\in \CC$,
e.g.\ $\CC=\Loc/M$ from Example \ref{ex:LocOCAT}, then every choice of
state $\omega_t : \AAA(t)\to I$ on the corresponding order-reversing $\ast$-monoid
defines a state on $\AAA$ via pullback $\omega_c := \omega_t \,\AAA(\exists! : c\to t)$
along the unique $\CC$-morphism to $t$. The GNS representation for $\omega_t:\AAA(t)\to I$
then defines a $\ast$-representation of $\AAA$ on a common inner product
space $(\AAA(t),\langle\cdot,\cdot\rangle_t)$. Such $\ast$-representations
are typically used for Haag-Kastler type algebraic 
quantum field theories on $\Loc/M$, cf.\ \cite{HaagKastler}.
\end{ex}

\paragraph*{$\E_\infty$-resolution and homotopy algebraic quantum field theories:}
The results of this section generalize to {\em homotopy algebraic quantum field theories}
\cite{BeniniSchenkelWoikehomotopy}. These are homotopy algebras over the colored operad
$\O_{(\CC,\perp)}$ in the symmetric monoidal {\em model} category $\Ch_\bbC$
of chain complexes of complex vector spaces. Concretely, we shall discuss the 
$\Sigma$-cofibrant resolution $w: \O_{(\CC,\perp)} \otimes \E_\infty \to \O_{(\CC,\perp)}$ 
obtained by the component-wise tensoring of the colored operad $\O_{(\CC,\perp)}$ 
and the Barratt-Eccles operad $\E_\infty$ from \cite{BergerFresse}. 
Algebras over the colored operad $\O_{(\CC,\perp)} \otimes \E_\infty $ 
play a prominent role in formalizing quantum gauge theories,
see \cite{BeniniSchenkelWoikehomotopy} for details. 
\sk

As a first step, we shall equip the {\em simplicial} Barratt-Eccles operad 
$\E_\infty^\sSet$ with a $\ast$-structure. Transfer along the 
normalized chains functor $N_\ast: \sSet \to \Ch_\bbC$ then will define a
$\ast$-structure on the operad $\E_\infty = N_\ast(\E_\infty^\sSet)$ in $\Ch_\bbC$.
Recall from e.g.\ \cite{BergerFresse} that the
simplicial set of $n$-ary operations in $\E_\infty^\sSet$ 
is the nerve of the action groupoid $\Sigma_n//\Sigma_n$. Explicitly, 
$\E_\infty^\sSet(n)_k := \Sigma_n^{\times k+1}$ 
is the set of $n$-ary operations of degree $k$. 
Consider now the trivial involutive symmetric monoidal category 
$(\sSet, \ID_\sSet, \id_{\ID_\sSet})$ of simplicial sets. 
We endow $\E_\infty^\sSet$ with a $\ast$-involution similar to that 
on the associative operad $\mathsf{As}$ in Example \ref{ex:reversal}, 
see also \eqref{eqn:AQFTast}. Explicitly, we define 
$\ast_\E: \E_\infty^\sSet \to \E_\infty^\sSet$ as the map that sends a tuple 
$(\sigma_0,\ldots,\sigma_k) \in \Sigma_n^{\times n+1}$ 
to $(\rho_n\sigma_0,\ldots,\rho_n\sigma_k) \in \Sigma_n^{\times n+1}$, 
where $\rho_n \in \Sigma_n$ is the order-reversal permutation 
from Example \ref{ex:Cprofiles}. Clearly, this provides 
a $\ast$-object structure on the underlying symmetric sequence, 
whose compatibility with the operadic composition and unit follows 
from elementary properties of the permutation group. 
\sk

Consider now the involutive symmetric monoidal category 
$(\Ch_\bbC,\overline{(-)},\id_{\ID_{\Ch_\bbC}})$ of chain complexes of complex vector spaces, 
obtained similarly to Examples \ref{ex:vec} and \ref{ex:vec3}. 
We equip the symmetric monoidal normalized chains functor 
$N_\ast: \sSet \to \Ch_\bbC$ with the structure 
of an involutive symmetric monoidal functor 
$(N_\ast,\nu): (\sSet,\ID_\sSet,\id_{\ID_\sSet}) \to 
(\Ch_\bbC,\overline{(-)},\id_{\ID_{\Ch_\bbC}})$ by declaring 
$\nu_X: N_\ast(X) \to \overline{N_\ast(X)}$ to 
act by complex conjugation on $\bbC$-valued chains in a simplicial set $X$. 
We define the Barratt-Eccles $\ast$-operad $\E_\infty$ in $\Ch_\bbC$
by applying the involutive symmetric monoidal functor $(N_\ast,\nu)$
to the $\ast$-operad $\E_\infty^\sSet$ in simplicial sets. 
Combining this with the colored $\ast$-operad structure from Proposition 
\ref{propo:AQFTastOp}, one immediately obtains the following result.
\begin{propo}
The component-wise tensor product of the $\ast$-involutions on $\O_{(\CC,\perp)}$ 
and $\E_\infty$ defines a colored $\ast$-operad structure on $\O_{(\CC,\perp)} \otimes \E_\infty $.
\end{propo}

\begin{rem}
Similarly to Remark \ref{rem:revMon}, the $\ast$-involution
on the Barratt-Eccles operad $\E_\infty$ considered above is not the only one. 
For example, one could replace order-reversal permutations 
by identity permutations. Our choice is motivated by the fact that
every $\ast$-algebra over $\E_\infty$ (in our sense) 
has an underlying order-reversing differential graded $\ast$-algebra.
This is a consequence of the evident $\ast$-operad
inclusion $\mathsf{As}\to \E_\infty$, where $\mathsf{As}$ carries the 
order-reversing $\ast$-structure from Example \ref{ex:reversal}.
\end{rem}


\section*{Acknowledgments}
We thank the anonymous referees for useful comments that helped us to improve this manuscript.
We also would like to thank John Barrett and 
Christoph Schweigert for useful discussions and comments.
The work of M.B.\ is supported by a research grant funded by 
the Deutsche Forschungsgemeinschaft (DFG, Germany). 
A.S.\ gratefully acknowledges the financial support of 
the Royal Society (UK) through a Royal Society University 
Research Fellowship, a Research Grant and an Enhancement Award. 
L.W.\ is supported by the RTG 1670 ``Mathematics inspired 
by String Theory and Quantum Field Theory''.


\end{document}